\documentclass{amsart}
\usepackage{amssymb,amsmath,amstext,amsthm,amsfonts,amscd}
\usepackage{latexsym}
\usepackage{enumerate}
\usepackage{graphicx}
\newtheorem{theorem}{Theorem}
\newtheorem{proposition}[theorem]{Proposition}
\newtheorem{corollary}[theorem]{Corollary}
\newtheorem{definition}[theorem]{Definition}
\newtheorem{lemma}[theorem]{Lemma}

\newcommand{\Z}{\mathbb{Z}}
\newcommand{\Q}{\mathbb{Q}}
\newcommand{\R}{\mathbb{R}}
\newcommand{\C}{\mathbb{C}}
\newcommand{\N}{\mathbb{N}}
\newcommand{\LN}{\mathcal{L}}
\newcommand{\IN}{\mathcal{I}}

\newcommand{\ind}{\mathbf{1}}
\newcommand{\E}{\mathcal{E}}

\newcommand{\DD}{\mathbb{D}}
\newcommand{\Co}{\mathcal{C}}
\newcommand{\PP}{\mathbb{P}}
\newcommand{\Di}{\mathcal{D}}
\newcommand{\Dio}{\mathcal{D}^\circ}
\newcommand{\Dit}{\widetilde{\mathcal{D}}}
\newcommand{\Ds}{\mathcal{T}^{\mathrm{ext}}}
\newcommand{\D}{\widetilde{\mathcal{D}}^{\C^*}}
\newcommand{\G}{\mathfrak{G}}
\newcommand{\e}{\mathrm{e}}
\newcommand{\re}{\mathrm{Re}}
\newcommand{\im}{\mathrm{Im}}
\newcommand{\Arg}{\mathrm{Arg}}
\newcommand{\res}{\mathrm{Res}}
\newcommand{\Aut}{\mathrm{Aut}}
\newcommand{\Tr}{\mathrm{Tr}}
\newcommand{\ie}{{\it{i.$\,$e.\ }}}

\begin{document}
\title{On certain zeta integral: transformation formula}
\author{Milton Espinoza}
\subjclass[2010]{Primary 11M32, 11M35, Secondary 11R42, 11F20}
\keywords{Multiple zeta functions, zeta integrals, Kronecker limit formulas, special values, Lambert series, multiple gamma function}
\email{milton.espinoza@uv.cl}
\address{Instituto de Matem\'aticas, 
   Facultad de Ciencias, 
   Universidad de Valpara\'iso, 
   Gran Breta\~na 1091, 3er piso, Valpara\'iso,
   Chile}
\thanks{I am grateful for the generous support of CONICYT BECAS CHILE 74150071 and the Max-Planck-Institut f\"ur Mathematik}

\begin{abstract}
We introduce an ``$L$-function'' $\LN$ built up from the integral representation of the Barnes' multiple zeta function $\zeta$. Unlike the latter, $\LN$ is defined on a domain equipped with a non-trivial action of a group $G$. Although these two functions differ from each other, we can use $\LN$ to study $\zeta$. In fact, the transformation formula for $\LN$ under $G$-transformations provides us with a new perspective on the special values of both $\zeta$ and its $s$-derivative. 

In particular, we obtain Kronecker limit formulas for $\zeta$ when restricted to points fixed by elements of $G$. As an illustration of this principle, we evaluate certain generalized Lambert series at roots of unity, establishing pertinent algebraicity results. Also, we express the Barnes' multiple gamma function at roots of unity as a certain infinite product.

It should be mentioned that this work also considers twisted versions of $\zeta$.
\end{abstract}

\maketitle

\section*{Introduction}

Our starting point is the \textit{Barnes' multiple zeta function} defined via the Dirichlet series
\begin{align}\label{Z-series}
\zeta_N(s,w,a):=\sum_{m\in\N^N}(w+m\cdot a)^{-s}.
\end{align}
Here $N$ is a positive integer called \textit{the dimension of the zeta function}, and the sum runs over all the $N$-tuples of non-negative integers. The parameters $s$, $w$, and the entries $a_\ell$ of the $N$-dimensional (row) vector $a\in\C^N$ are complex numbers. We denote $x\cdot y$ the standard dot product between $x$ and $y$, \ie
\begin{align*}
x\cdot y:=x_1y_1+\dots +x_Ny_N \qquad\qquad (x,y\in\C^N).
\end{align*}
After choosing a suitable branch of the logarithm to define complex powers, it can be shown that $\zeta_N(s,w,a)$ converges absolutely whenever $\re(s)>N$, $\re(w)>0$, and $\re(a_\ell)>0$ for all $\ell\in\{1,\dots,N\}$. If we consider the tube-like domain 
\begin{align*}
\mathcal{T}_N^+=\big\{(w,a)\in\C\times\C^N \ \big| \ \re(w)>0, \ \re(a_\ell)>0, \ 1\leq \ell\leq N\big\},
\end{align*}
then $\zeta_N$ is a holomorphic function on the Cartesian product $\{s\in\C \ | \ \re(s)>N\}\times \mathcal{T}_N^+$,
and the map $s\mapsto \zeta_N(s,w,a)$ can be meromorphically continued to $\C$ for fixed $(w,a)\in \mathcal{T}_N^+$. 

Historically, $\zeta_N$ was introduced by Barnes \cite{Ba} at the beginning of the 20th century as a natural generalization of the Hurwitz zeta function. Later it was considered by Shintani \cite{Sh1} \cite{Sh2} who showed that it arises as a critical term in the evaluation of Hecke $L$-series at $s=1$. Also, a modern and more general treatment of $\zeta_N$ can be found in \cite{Ru} and \cite{Fr-Ru}. 

The analytic $s$-continuation of $\zeta_N(s,w,a)$ can be performed in two steps, as it is detailed in \cite[\S 2.4]{Hi}. First, one obtains an integral representation in the domain of absolute convergence of the series \eqref{Z-series}. As it has been the case of its younger avatars, this integral representation comes from applying the Mellin transform to some suitable test function, namely
\begin{equation*}
F_N(u,w,a):=\frac{\e^{-uw}}{(1-\e^{-ua_1})(1-\e^{-ua_2})\dots(1-\e^{-ua_N})} \qquad\qquad (u\in\C).
\end{equation*}
Then one converts the resulting integral into a contour one by means of the Hankel contour, which gives a meromorphic function on $\C$ having at most finitely many simple poles for each fixed $(w,a)\in \mathcal{T}_N^+$. In particular, the non-positive integers are regular, and we can evaluate $\zeta_N(s,w,a)$ at these points in closed form. The special values $\zeta_N(-k,w,a)$, $k\in\N$, are very nice functions of $(w,a)$ whose domains can be easily extended. They are polynomials in $w$ with coefficients in the field $\Q(a)$ generated by the entries of $a$, and furthermore they satisfy several symmetries that can be synthesized in a very simple transformation formula under the action of a group $G_N$ (defined below). The above suggests the following heuristic observation: $\zeta_N$ is actually a function of the three variables $s$, $w$, and $a$, but it gets into its best form only if one fixes either $s$ or $(w,a)$. Unfortunately, this is not good enough for some applications. For instance, if we wanted to study the $s$-derivative of $\zeta_N(s,w,a)$ at non-positive integers, we would also be interested in knowing the behavior of $\zeta_N(s,w,a)$ as $s$ varies around these points, and hence it would be desirable to keep track of any symmetry corresponding to the $G_N$-action on the variable $(w,a)$ even when $s$ is treated as another variable.

The aim of this article is to introduce a function $\LN_N(s,w,a)$ such that 
\begin{itemize}
\item the map $s\mapsto \LN_N(s,w,a)$ defines a meromorphic function on $\C$ for fixed $(w,a)$,
\item the variable $(w,a)$ ranges over a domain equipped with the action of $G_N$,
\item the transformation formula for $\LN_N(s,w,a)$ under $G_N$-transformations holds for any $s$, and
\item we can study the Barnes' multiple zeta function using $\LN_N$ and its symmetries.
\end{itemize}
In order to describe the group $G_N$ in consideration, let $\{\pm1\}^N$ be the $N$-ary Cartesian power of the multiplicative group of order 2, let $S_N$ be the symmetric group on $N$ elements, and let $\C^*$ be the multiplicative group of nonzero complex numbers. Then $G_N$ is isomorphic to the product $\big(\{\pm1\}^N\rtimes_\varphi S_N\big)\times\C^*$, where $\{\pm1\}^N\rtimes_\varphi S_N$ denotes the (outer) semidirect product with respect to the homomorphism
\begin{align*}
\varphi:S_N\to\Aut\big(\{\pm1\}^N\big), \qquad\varphi_\sigma(\varepsilon_1,\dots,\varepsilon_N)=(\varepsilon_{\sigma^{-1}(1)},\dots,\varepsilon_{\sigma^{-1}(N)}) \qquad (\sigma\in S_N).
\end{align*}
Moreover, we will also allow our functions to be \textit{twisted}, \ie we will introduce an extra parameter $\theta\in\R^{N}$ such that the functions $\LN_N(s,w,a,\theta)$ will handle Dirichlet series of the form 
\begin{align*}
\zeta_N(s,w,a,\theta):=\sum_{m\in\N^N}\e(m\cdot \theta)(w+m\cdot a)^{-s} \qquad\qquad \big(\re(s)>N\big).
\end{align*}
Here, and from now on, we write $\e(z):=\e^{2\pi iz}$ for all $z\in\C$. Thus all of our results will consider functions in this generality, the Barnes case being the non-twisted specialization $\theta\in\Z^{N}$.

We now briefly outline the steps in the construction of $\LN_N$. First we take the integral representation yielding the analytic continuation of $\zeta_N(s,w,a,\theta)$, and we note that it gives a nice function $\LN_N(s,w,a,\theta)$ of $s$, which is actually defined on a larger set $\Di_N$ of elements $(w,a,\theta)$ than the one considered for the sake of absolute convergence of the series representation. Furthermore, the group $\{\pm1\}^N\rtimes_\varphi S_N$ acts on $\Di_N$, and the corresponding transformation formula for $\LN_N(s,w,a,\theta)$ generalizes well-known symmetry relations satisfied by Bernoulli polynomials. Next, we extend $\Di_N$ to a domain $\Dit_N$ on which both $s\mapsto\LN_N(s,w,a,\theta)$ and the action of $\{\pm1\}^N\rtimes_\varphi S_N$ can be defined through a limiting process. In order to include also the action of $\C^*$, we take a suitable subset $\D_N$ of $\Dit_N$. Then the desired function follows by taking the restriction of $s\mapsto\LN_N(s,w,a,\theta)$ to $\D_N$. This function differs from $\zeta_N$, but it can be used to compute the latter because (i) they coincide in a certain part of their domains, and (ii) such part is, up to a well-controlled set of measure zero, a fundamental domain for the action of $\{\pm1\}^N$ on $\D_N$. Finally, the transformation formula for $\LN_N(s,w,a,\theta)$ under the action of the whole group $\big(\{\pm1\}^N\rtimes_\varphi S_N\big)\times\C^*$ provides us with a new perspective on the special values of both $\zeta$ and its $s$-derivative. As a consequence, we evaluate certain series of the form
\begin{align*}
\sum_n a_n\cdot\Big(\frac{q_1^n}{1-\xi_1q_1^n}\Big)\dots\Big(\frac{q_m^n}{1-\xi_mq_m^n}\Big),
\end{align*}
establishing algebraicity results (Corollaries \ref{Cor.Ex.1} and \ref{Cor.Ex.2}), and we express the Barnes' multiple gamma function 
\begin{align}\label{Mult.Gamma}
\Gamma_N(w,a):=\mathrm{exp}\left(\frac{d}{ds}\zeta_N(s,w,a)\Big|_{s=0}\right) \qquad\qquad \big((w,a)\in \mathcal{T}_N^+\big)
\end{align}
at roots of unity as an infinite product (Corollary~\ref{Cor.Ex.2}).

In short, our approach exploits properties and relations offered by the test function $F_N$, and how the integration process yielding the analytic continuation of the Barnes' multiple zeta function interacts with them. This can be seen as part of a general philosophy leading to fruitful and diverse applications. For instance, Hirose and Sato \cite{Hi-Sa} gave functional equations for \textit{normalized Shintani $L$-functions}, \ie zeta-integrals attached to certain twisted Dirichlet series parameterized by real invertible matrices, and although they have been guided by the same idea, their results are quite different in nature.

This article is divided into four sections. The first one is devoted to the statement of the main results. We start by summarizing the analytic continuation of $\zeta_N(s,w,a,\theta)$ in both variables $s$ and $(w,a,\theta)$, and then we introduce the required notation in order to establish finally the most important outcomes of the paper. Except for Corollaries \ref{Cor.Ex.1} and \ref{Cor.Ex.2}, no proofs are given here, as we include them instead in the following sections. More precisely, in the second section we detail the action of $G_N$ on $\D_N$. In the third one we study the integral representation of $\LN_N$, while in the fourth one we elaborate on the extension of $\LN_N$ to $\Dit_N$ and its transformation formula corresponding to the $G_N$-action on $\D_N$. Also, the reader will find a summary of the most relevant domains considered in this work at the end of the article.

\section{Main results}

\subsection{Analytic continuation of $\zeta_N$}

We briefly sum up some well-known features of the analytic continuation of the Barnes' multiple zeta function. Let $N$ be a positive integer. In order to take the parameter $\theta\in\R^N$ into account, we set
\begin{align*} 
\mathcal{T}_N^+:=\big\{(w,a,\theta)\in\C\times\C^N\times\R^N \, \big| \, \re(w)>0, \ \re(a_\ell)>0, \ 1\leq \ell\leq N\big\}.
\end{align*}
Since this domain is a natural generalization of the $\mathcal{T}_N^+$ given in the introduction, we will keep this notation from now on. Also, we define $H_N:=\{s\in\C\,|\,\re(s)>N\}$ and $p$ as the projection of $\C\times\C^N\times\R^N$ onto the first two coordinates, \ie
\begin{align}
\label{proj.p} p:\C\times\C^{N}\times\R^{N}\to \C\times\C^{N}, \qquad\qquad p(w,a,\theta):=(w,a).
\end{align}

Using the principal branch of the logarithm to define complex powers, we consider the (possibly twisted) Barnes' multiple zeta function  
\begin{align}
& \label{Z-twistedseries} \zeta_N:H_N\times \mathcal{T}_N^+\to \C, \qquad\qquad \zeta_N(s,w,a,\theta):=\sum_{m\in\N^N}\e(m\cdot \theta)(w+m\cdot a)^{-s}.
\end{align}
It is well-defined and holomorphic as a function of $(s,w,a)\in H_N\times p(\mathcal{T}_N^+)$ for each fixed $\theta\in\R^{N}$. Using the test function
\begin{align}\label{F_N}
F_N(u,w,a,\theta):=\e^{-uw}\prod_{\ell=1}^N\big(1-\e(\theta_\ell)\e^{-ua_\ell}\big)^{-1}, 
\end{align}
it also admits the integral representation 
\begin{align}
\nonumber \zeta_N(s,w,a,\theta)=\frac{1}{\Gamma(s)}\int_0^\infty F_N(u,w,a,\theta)u^{s-1}du, \qquad (s,w,a,\theta)\in H_N\times \mathcal{T}_N^+. 
\end{align}
Then we can transform the above integral into a contour one by means of the Hankel contour $C(\epsilon)$ for sufficiently small $\epsilon>0$. Recall that $C(\epsilon)$ is defined as the counterclockwise oriented path consisting of the interval $[\epsilon,+\infty)$, and the circle of radius $\epsilon$ centered at the origin followed by the same interval. Hence the analytic $s$-continuation of $\zeta_N$ is given by
\begin{align}
\label{anal.cont.zeta} \zeta_N(s,w,a,\theta)=\frac{1}{\Gamma(s)(\e(s)-1)}\int_{C(\epsilon)}F_N(u,w,a,\theta)u^{s-1}du, \qquad (w,a,\theta)\in \mathcal{T}_N^+.
\end{align}
It can be shown that it is independent of $\epsilon$, and that the integral defines an entire function of $s$ on $\C$ having zeros at the integers greater than $N$. Therefore we have an extension $\zeta_N:(\C\smallsetminus \{1,\dots,N\})\times \mathcal{T}_N^+\to\C$,
for which the elements in $\{1,\dots,N\}$ are, at most, simple poles for each fixed $(w,a,\theta)\in \mathcal{T}_N^+$. 

Now here is the crux: to extend the domain of $\zeta_N$ with respect to $\mathcal{T}_N^+$, we can proceed by considering either the series or the integral representation. The former way goes back to Barnes \cite{Ba} and it comes from noticing that some rotations of the parameters $w$ and $a$ do not affect the absolute convergence of the series in \eqref{Z-twistedseries}. Indeed, for every angle $\omega\in(-\pi/2,\pi/2)$, define
\begin{align*}
& \mathcal{T}_N^+(\omega):=\big\{(w,a,\theta)\in\C\times\C^{N}\times\R^{N} \, | \, (\e^{i\omega}w,\e^{i\omega}a,\theta)\in \mathcal{T}_N^+\big\},
\end{align*}
and then define the function
\begin{align*}
& \zeta_N^{\omega}:(\C\smallsetminus \{1,\dots,N\})\times \mathcal{T}_N^+(\omega)\to\C, \qquad
\zeta_N^{\omega}(s,w,a,\theta):=\e^{is\omega}\zeta_N(s,\e^{i\omega}w,\e^{i\omega}a, \theta).
\end{align*}
Writing
\begin{equation*}
\Ds_N:=\bigcup_{\omega\in(-\pi/2,\pi/2)}\mathcal{T}_N^+(\omega),
\end{equation*}
we have that $\Ds_N$ is simply connected, contains $\mathcal{T}_N^+$, and permits the analytic continuation
\begin{align}
\nonumber \zeta_N:(\C\smallsetminus \{1,\dots,N\})\times \Ds_N\to\C, \qquad & \zeta_N(s,w,a,\theta):=\zeta_N^{\omega}(s,w,a,\theta),\\ 
\label{series anal. cont.} & (w,a,\theta)\in \mathcal{T}_N^+(\omega), \ \mathrm{some} \ \omega\in(-\pi/2,\pi/2),
\end{align}
as it is detailed in \cite[\S 6]{Fr-Ru}. One remarkable trait of the above $\zeta_N(s,w,a,\theta)$ is that it is a holomorphic function of $(s,w,a)\in (\C\smallsetminus \{1,\dots,N\})\times p(\Ds_N)$ for each fixed $\theta\in\R^{N}$. However, $\Ds_N$ has the drawback of being \textit{non-symmetric}, in the sense that it compels us to work with column matrices $a$ having entries in a half-plane. Furthermore, some half-planes have had to be dismissed in order to avoid multivaluedness. In this article, we extend the domain of $\zeta_N$ with respect to $\mathcal{T}_N^+$ by considering the integral representation \eqref{anal.cont.zeta}.

\subsection{Statement of the main results} Let us start by fixing some notation. Let $\R_+$ be the set of positive real numbers. Define the convex cone 
\begin{align}\label{Co}
\Co:=\big\{z\in\C\,\big|\,\re(z)>0 \quad \text{or} \quad z\in i\cdot\R_+\big\},
\end{align}
and note that its interior $\Co^\circ$ amounts to the right half-plane and that $\Co\cup\{0\}\cup-\Co=\C$.  

Let $N$ be a positive integer, set $P_N:=\{1,2,\dots,N\}$, and let $\PP_N$ be the power set of $P_N$. Since every $v\in\C^{N}$ is actually a function $v:P_N\to\C$, the inverse image $v^{-1}[S]\in\PP_N$ of $S$ under $v$ is defined for any $S\subseteq\C$. Then we define 
\begin{align}\label{good.domain}
&\DD_N:=\big\{(w,a,\theta)\in\C\times \C^{N}\times \R^{N} \, \big| \, a^{-1}[0]\subseteq\theta^{-1}[\R\smallsetminus\Z]\big\}.
\end{align}
Also, we define the \textit{trace function}
\begin{align}\label{trace}
\Tr:\C^N\times \PP_N\to\C,\qquad\qquad \Tr(v,\Lambda):=\sum_{\ell\in\Lambda}v_\ell.
\end{align}

Let $\mathcal{F}(\DD_N,\C)$ be the set of all complex-valued functions on $\DD_N$, and consider the distinguished element
\begin{align}\label{pi}
\pi:\DD_N\to\C,\qquad\qquad\pi(w,a,\theta):=w-\Tr(a,a^{-1}[-\Co]),
\end{align}
which is discontinuous at points $(w,a,\theta)$ where some $a_\ell$ is a nonzero purely imaginary complex number. Let $\Aut(\DD_N)$ be the group of all homeomorphisms of $\DD_N$ onto itself. Then $\Aut(\DD_N)$ acts on $\mathcal{F}(\DD_N,\C)$ by composition on the right, and we can consider the stabilizer subgroup
\begin{align*}
\Aut(\DD_N)_\pi:=\{g\in \Aut(\DD_N)\,|\,\pi g=\pi\}
\end{align*}
of $\Aut(\DD_N)$ with respect to $\pi$. The above induces an action of $\Aut(\DD_N)_\pi$ on the set $\mathcal{F}(\pi^{-1}[\Co],\C)$ of all complex-valued functions on $\pi^{-1}[\Co]$. From now on we set
\begin{align}\label{Dit}
\Dit_N:=\pi^{-1}[\Co]=\big\{(w,a,\theta)\in\DD_N\,|\,\pi(w,a,\theta)\in\Co\big\}.
\end{align}

Now we describe a finite subgroup of $\Aut(\DD_N)_\pi$ explicitly. For each $\Lambda\in\PP_N$, we denote $d(\Lambda)$ the $N\times N$ diagonal matrix whose $(\ell,\ell)$-entry equals either $-1$ if $\ell\in\Lambda$, or $1$ otherwise. Then we define the function
\begin{align}
T_\Lambda=T_{\Lambda,N}:\DD_N\to \DD_N, \qquad
  T_\Lambda(w,a,\theta):=\big(w-\Tr(a,\Lambda) \, , \,  ad(\Lambda) \, , \, \theta d(\Lambda)\big).
  \label{T_Lambda} 
\end{align} 
Next, for each $\sigma\in S_N$, consider the corresponding permutation matrix $r(\sigma)$, \ie the $N\times N$ matrix whose $\ell$-th column equals the $\sigma^{-1}(\ell)$-th column of the identity matrix of size $N$. Then we define the function 
\begin{align}\label{R_sigma}
R_\sigma=R_{\sigma,N}:\DD_N\to\DD_N, \qquad\qquad R_\sigma(w,a,\theta):=\big(w \, , \, a r(\sigma) \, , \, \theta r(\sigma)\big).
\end{align}

\begin{proposition}\label{prop.R_sigma.T_Lambda}
The $T_\Lambda$ and $R_\sigma$ $(\Lambda\in\PP_N,\,\sigma\in S_N)$ defined above lie in $\Aut(\DD_N)_\pi$. Furthermore, if we let  $\mathfrak{T}_N$ and $\mathfrak{R}_N$ be the groups generated by the $T_\Lambda$ and $R_\sigma$ $(\Lambda\in\PP_N,\,\sigma\in S_N)$ respectively, then the set $\mathfrak{T}_N\mathfrak{R}_N$ equipped with the operation
\begin{align*}
T_{\Lambda_1}R_{\sigma_1}\cdot T_{\Lambda_2}R_{\sigma_2}=T_{\Lambda_1\oplus\sigma_1(\Lambda_2)}R_{\sigma_1\sigma_2}
\qquad\qquad (\Lambda_i\in\PP_N,\,\sigma_i\in S_N,\,  i=1,2)
\end{align*}
is a group isomorphic to the outer semidirect product $\{\pm1\}^N\rtimes_\varphi S_N$ with respect to
\begin{align*}
\varphi:S_N\to\Aut\big(\{\pm1\}^N\big), \qquad\varphi_\sigma(\varepsilon_1,\dots,\varepsilon_N)=(\varepsilon_{\sigma^{-1}(1)},\dots,\varepsilon_{\sigma^{-1}(N)}).
\end{align*}
Here and from now on the symbol $\oplus$ denotes the symmetric difference of sets.
\end{proposition}

In order to include the action of $\C^*$, we define the following transformations. For each $\alpha\in\C^*$, consider the function
\begin{align}\label{M_alpha}
M_\alpha=M_{\alpha,N}:\DD_N\to\DD_N, \qquad\qquad M_\alpha(w,a,\theta):=(\alpha w,\alpha a, \theta),
\end{align}
which is clearly a homeomorphism. The group $\mathfrak{M}_N$ generated by the $M_\alpha$ ($\alpha\in\C^*$) is isomorphic to $\C^*$ in a natural way. Unlike the previous cases, $\mathfrak{M}_N$ is not contained in $\Aut(\DD_N)_\pi$, thus we consider instead the group
\begin{align*}
\G_N:=\big[\Aut(\DD_N)_\pi\cap \mathrm{N}(\mathfrak{M}_N)\big]\mathfrak{M}_N,
\end{align*}
where $\mathrm{N}(\mathfrak{M}_N)$ denotes the normalizer of $\mathfrak{M}_N$ in $\Aut(\DD_N)$. Hence $\G_N$ acts on the set $\mathcal{F}(\D_N,\C)$ by composition on the right, where
\begin{align}\label{D.final}
\D_N:=\bigcap_{\alpha\in\C^*}(\pi M_\alpha)^{-1}[\Co].
\end{align}

\begin{proposition}\label{prop.R_sigma.T_Lambda.M_alpha}
The $T_\Lambda$, $R_\sigma$, and $M_\alpha$ $(\Lambda\in\PP_N,\,\sigma\in S_N, \,\alpha\in\C^*)$ defined above lie in $\G_N$. Furthermore, the set $G_N:=\mathfrak{T}_N\mathfrak{R}_N\mathfrak{M}_N$ equipped with the operation
\begin{align*}
T_{\Lambda_1}R_{\sigma_1}M_{\alpha_1}\cdot T_{\Lambda_2}R_{\sigma_2}M_{\alpha_2}=T_{\Lambda_1\oplus\sigma_1(\Lambda_2)}R_{\sigma_1\sigma_2}M_{\alpha_1\alpha_2},
\end{align*}
for any $\Lambda_i\in\PP_N$, $\sigma_i\in S_N$, and $\alpha_i\in\C^*$ $(i=1,2)$, is a subgroup of $\G_N$ isomorphic to the product $\big(\{\pm1\}^N\rtimes_\varphi S_N\big)\times\C^*$.
\end{proposition}

We know that $G_N$ acts on $\DD_N$ by homeomorphisms, thus it seems natural to ask about the sets 
\begin{align*}
\DD_N^g:=\big\{(w,a,\theta)\in\DD_N\,\big|\,g(w,a,\theta)=(w,a,\theta)\big\} \qquad\qquad (g\in G_N)
\end{align*}
of points fixed by elements of $G_N$. In fact, the next proposition shows that these sets can be characterized in terms of eigenspaces of certain matrices. For any $N\times N$ complex matrix $A$ and any $\lambda\in\C$, we denote $E_\lambda[A]:=\{v\in\C^N\,|\,vA=\lambda v\}$ the eigenspace of $A$ associated with $\lambda$. Note that $E_\lambda[A]$ is non-trivial if and only if $\lambda$ is an eigenvalue of $A$. 

\begin{proposition}\label{prop.fixed.points}
Let $\Lambda\in\PP_N$, $\sigma\in S_N$, and $\alpha\in\C^*$. Set $g:=T_\Lambda R_\sigma M_\alpha$ and $A:=r(\sigma)d(\Lambda)$. 
\begin{enumerate}[$(i)$]
\item\label{prop.fixed.points.alpha=1} Suppose that $\alpha=1$. Then 
\begin{align*}
\DD_N^g=\big\{(w,a,\theta)\in\DD_N\,\big|\, \Tr(a,\Lambda)=0 \quad \text{and} \quad a, \theta\in E_1[A]\big\}.
\end{align*}
Furthermore, if $(w,a,\theta)\in\DD_N^g$, then $\Tr(\theta,\Lambda)=0$ and $|\Lambda|$ is even.
\item\label{prop.fixed.points.any.alpha} Suppose that $\alpha\not=1$. Then 
\begin{align*}
\DD_N^g=\Big\{(w,a,\theta)\in\DD_N\,\Big|\,w=\frac{1}{2}\Tr(a,P_N), \quad a\in E_{\alpha^{-1}}[A], \quad \text{and} \quad \theta\in E_1[A]\Big\}
\end{align*}
and we have the inclusion $\DD_N^g\subseteq \D_N$. Furthermore, if there exists $(w,a,\theta)\in\DD_N^g$ with $a\not=0$, then $\alpha$ is a root of unity.
\end{enumerate} 
\end{proposition}

Now we address the extension of $\zeta_N:(\C\smallsetminus P_N)\times \mathcal{T}_N^+\to\C$ to the domain $(\C\smallsetminus P_N)\times \Dit_N$ by using \eqref{anal.cont.zeta}. This is not an analytic continuation in the variable $(s,w,a)$, but the resulting extension is actually holomorphic up to a well-controlled set of measure zero. 

Note that the size of the ``sufficiently small $\epsilon>0$'' considered in \eqref{anal.cont.zeta} depends on $(w,a,\theta)$, and that there are infinitely many choices of it for each such triple. To control this situation better, we record the following definition. 

\begin{definition}\label{Epsilon}
Let $\E=\E_N$ be the set of all functions $\epsilon:\DD_N\to\R_+$ satisfying the following condition: for all $(w,a,\theta)\in \DD_N$ and all $\ell\in P_N$, the equation $\e(\theta_\ell)=\e^{ua_\ell}$ has no solutions $u\in\C$ with $0<|u|\leq \epsilon(w,a,\theta)$. Then we define formally
\begin{align*}
\LN_{N,\epsilon}(s,w,a,\theta):=\frac{1}{\Gamma(s)(\e(s)-1)}\int_{C(\epsilon(w,a,\theta))}F_N(u,w,a,\theta)u^{s-1}du 
\end{align*}
for each $\epsilon\in\E$ and each $(s,w,a,\theta)\in\C\times\DD_N$. Here we consider 
\begin{align*}
u^{s-1}=\e^{(s-1)(\log|u|+i\Arg(u))},
\end{align*}
where $\Arg(u)=0$ when $u$ varies on the negatively oriented interval $(0,+\infty)$, $0<\Arg(u)< 2\pi$ when $u$ goes counterclockwise around the origin, and $\Arg(u)=2\pi$ when $u$ varies on the positively oriented interval $(0,+\infty)$.
\end{definition}

\begin{theorem}\label{Thm.constr.L}
The function $\LN_N:(\C\smallsetminus P_N)\times \Dit_N\to\C$ given by
\begin{align}\label{L.final}
\LN_N(s,w,a,\theta):=\lim_{\substack{\omega\to0\\ \omega<0}}\LN_{N,\epsilon}\big(s,M_{\e^{i\omega}}(w,a,\theta)\big)
\end{align}
is well-defined and independent of the choice of $\epsilon\in\E$. Furthermore, it satisfies the following properties.
\begin{enumerate}[$(i)$]
\item\label{Thm.constr.L.mero.s} For each fixed $(w,a,\theta)\in \Dit_N$, the map $s\mapsto \LN_N(s,w,a,\theta)$ defines a meromorphic function on $\C$ having at most simple poles at $P_N$.
\item\label{Thm.constr.L.mero.swa} For each fixed $\theta\in\R^N$ and any $\Lambda\in\PP_N$, the map $(s,w,a)\mapsto \LN_N(s,w,a,\theta)$ defines a holomorphic function on $(\C\smallsetminus P_N)\times p\big(T_\Lambda(\mathcal{T}_N^+)\big)$.
\item\label{Thm.constr.L.rest.} For every $(s,w,a,\theta)\in (\C\smallsetminus P_N)\times \Co\times\Co^N\times\R^N$, we have that $\LN_N(s,w,a,\theta)=\zeta_N(s,w,a,\theta)$.
\end{enumerate}
\end{theorem}

There are two choices in \eqref{L.final}: $\epsilon$ and $\lim_{\substack{\omega\to0^-}}$. As the above theorem claims, the former is irrelevant since it does not affect the outcome. In contrast, the way of how $\omega$ approaches 0 has consequences on the values $\LN_N(s,w,a,\theta)$ when some $a_\ell$ is a nonzero purely imaginary complex number. Therefore, the function $\LN_N$ in \eqref{L.final} is a sort of \textit{principal value for $\LN_N$}. 

Now we state the second theorem of this article. Recall that $G_N:=\mathfrak{T}_N\mathfrak{R}_N\mathfrak{M}_N$,
\begin{align*}
F_N(u,w,a,\theta):=\e^{-uw}\prod_{\ell=1}^N\big(1-\e(\theta_\ell)\e^{-ua_\ell}\big)^{-1},
\end{align*}
and $\D_N\subseteq \Dit_N$ (see \eqref{Dit} and \eqref{D.final}). Let $\arg:\C^*\to(-\pi,\pi]$ be the principal argument function, and let $\mathrm{sgn}:\R\to\{-1,0,1\}$ be the usual sign function. For any nonzero angle $\psi\in[-\pi,\pi]$ and any $(w,a,\theta)\in\DD_N$, we define $\mathcal{P}_\psi(a,\theta)$ as the set of nonzero poles $u=u_0$ of $F_N(u,w,a,\theta)$ satisfying either $\arg(u_0)\in[0,\psi)$ if $\psi>0$, or $\arg(u_0)\in[\psi,0)$ if $\psi<0$. Also, for any $g=T_\Lambda R_\sigma M_\alpha\in G_N$, we define 
\begin{align}\label{J_g}
J_g(s,\theta):=(-1)^{|\Lambda|}\e\big(\Tr\big(\theta,\sigma^{-1}[\Lambda]\big)\big)\alpha^{-s} \qquad\qquad (s\in\C,\,\theta\in\R^N),
\end{align}
where we use the principal branch of the logarithm to define $\alpha^{-s}$.

\begin{theorem}\label{Thm.transf.form.}
The function $\LN_N:(\C\smallsetminus P_N)\times \D_N\to\C$ satisfies the following transformation formula under $G_N$-transformations. Let  $\Lambda\in\PP_N$, $\sigma\in S_N$, and $\alpha\in\C^*$. Set $g:= T_\Lambda R_\sigma M_\alpha$ and $\psi:=\arg(\alpha)$. Then
\begin{align}\label{Transf.form.LN}
\LN_N\big(s,g(w,a,\theta)\big)=J_g(s,\theta)\Big[\LN_N(s,w,a,\theta)+\rho_N^\psi(s,w,a,\theta)/\Gamma(s)\Big]
\end{align} 
for all $(s,w,a,\theta)\in (\C\smallsetminus P_N)\times \D_N$, where $\rho_N^\psi(s,w,a,\theta)$ is the $s$-analytic continuation of the function 
\begin{align}\nonumber
\rho_N^\psi(s,w,a,\theta):=\mathrm{sgn}(-\psi)2\pi i\lim_{R\to\infty}\sum_{\substack{u_0\in\mathcal{P}_\psi(a,\theta)\\ |u_0|<R}}\res_{u=u_0}\big(F_N(u,w,a,\theta)u^{s-1}\big)
\end{align}
holomorphic for $\re(s)<1$. Here we use the principal branch of the logarithm to define both $\alpha^{-s}$ and $u^{s-1}$. 

Furthermore, for every $k\in\Z$, the factor $J_g(k,\theta)$ defines a 1-cocycle for the action of $G_N$ on the multiplicative group of everywhere nonzero functions in $\mathcal{F}(\DD_N,\C)$. As a consequence, for any $g, h\in G_N$ and any $\delta\in\D_N$, we have
\begin{align}\label{Rho.Trans.Form.}
&\rho_N^{\psi_g}\big(k,h(\delta)\big)=\\
\nonumber &J_h(k,\delta)\left[\rho_N^{\psi_{gh}}(k,\delta)-\rho_N^{\psi_{h}}(k,\delta)+\Big(\frac{\psi_g+\psi_h-\psi_{gh}}{2\pi}\Big)\cdot\mathrm{coeff}\big(F_N(u,\delta),u^{-k}\big)\right]
\end{align}
for all $k\in\Z$ (except for possible poles), where $\psi_g$, $\psi_h$, and $\psi_{gh}$ are the angles associated with $g$, $h$, and $gh$ respectively. Here, and from now on, $\mathrm{coeff}(f(z),z^n)$ denotes the coefficient of $z^n$ in the Laurent expansion of $f(z)$ at $z=0$.
\end{theorem}

\begin{corollary}\label{Coro.trans.form.}
Let $\Lambda\in\PP_N$, $\sigma\in S_N$, and $\alpha\in\C^*$. Set $g:=T_\Lambda R_\sigma M_\alpha$ and $\psi:=\arg(\alpha)$. Then
\begin{align*}
\LN_N\big(-k,g(w,a,\theta)\big)=J_g(-k,\theta)\LN_N(-k,w,a,\theta)
\end{align*}
for all $k\in\N$ and all $(w,a,\theta)\in\D_N$. Likewise, for the $s$-derivative $\LN'_N$ we have that
\begin{align*}
&\LN'_N\big(-k,g(w,a,\theta)\big)=\\
&J_g(-k,\theta)\Big[\LN'_N(-k,w,a,\theta)-\log(\alpha)\LN_N(-k,w,a,\theta)+(-1)^kk!\rho_N^\psi(-k,w,a,\theta)\Big]
\end{align*}
for all $k\in\N$ and all $(w,a,\theta)\in\D_N$, where $\log(\alpha):=\log|\alpha|+i\psi$.
\end{corollary}

The above formulas unveil a straightforward relation between $\LN_N$ and $\rho_N$ when we restrict the elements $(w,a,\theta)\in\D_N$ to lie in the set $\DD_N^g$ of points fixed by some $g$. We can refer to them as Kronecker limit formulas, in the sense that they express the leading coefficient of $\LN_N$ at certain points in terms of $\rho_N$.

\begin{corollary}\label{Coro.trans.form.fix.points}
Let $\Lambda\in\PP_N$, $\sigma\in S_N$, and $\alpha\in\C^*$. Set $g:= T_\Lambda R_\sigma M_\alpha$ and $\psi:=\arg(\alpha)$. Assume that $(w,a,\theta)\in\D_N\cap\DD_N^g$. Then
\begin{align*}
\rho_N^\psi(s,w,a,\theta)=\Gamma(s)\Big[J_g(s,\theta)^{-1}-1\Big]\LN_N(s,w,a,\theta)
\end{align*}
for all $s\in\C\smallsetminus P_N$. In particular, we have the following formulas for $s=k\in\Z$.
\begin{enumerate}[$(i)$]
\item If $k\leq 0$ and $J_g(k,\theta)=1$, then 
\begin{align*}
\rho_N^\psi(k,w,a,\theta)=\frac{(-1)^k}{(-k)!}\log(\alpha)\LN_N(k,w,a,\theta).
\end{align*}
\item If $k\leq 0$ and $J_g(k,\theta)\not=1$, then $\LN_N(k,w,a,\theta)=0$ and
\begin{align*}
\rho_N^\psi(k,w,a,\theta)=\frac{(-1)^k}{(-k)!}\Big[J_g(k,\theta)^{-1}-1\Big]\LN'_N(k,w,a,\theta)
\end{align*}
\item If $k\in P_N$ and $J_g(k,\theta)=1$, then 
\begin{align*}
\rho_N^\psi(k,w,a,\theta)=(k-1)!\log(\alpha)\res_{s=k}\big(\LN_N(s,w,a,\theta)\big).
\end{align*}
\item If $k\in P_N$ and $J_g(k,\theta)\not=1$, then 
\begin{align*}
\res_{s=k}\big(\rho_N^\psi(s,w,a,\theta)\big)=(k-1)!\Big[J_g(k,\theta)^{-1}-1\Big]\res_{s=k}\big(\LN_N(s,w,a,\theta)\big).
\end{align*}
\item If $k\geq N+1$ and $J_g(k,\theta)=1$, then $\rho_N^\psi(k,w,a,\theta)=0$ and
\begin{align*}
\frac{d}{ds}\rho_N^\psi(s,w,a,\theta)\Big|_{s=k}=(k-1)!\log(\alpha)\LN_N(k,w,a,\theta).
\end{align*}
\item If $k\geq N+1$ and $J_g(k,\theta)\not=1$, then 
\begin{align*}
\rho_N^\psi(k,w,a,\theta)=(k-1)!\Big[J_g(k,\theta)^{-1}-1\Big]\LN_N(k,w,a,\theta).
\end{align*}
\end{enumerate}
\end{corollary}

\noindent \textbf{Example 1.} Let $N\geq3$ be an odd integer, let $\sigma$ be the $N$-cycle $(12\dots N)\in S_N$, let $\eta=\e(1/N)$, and set $g=R_\sigma M_\eta\in G_N$. Then it easily follows that $g$ fixes $\delta=(0,a,\theta)$, where
\begin{align*}
a=(\eta,\eta^{2},\dots,\eta^{N-1},1) \quad\qquad\text{and}\quad\qquad \theta=(c,c,\dots,c) \qquad (c\in\R).
\end{align*}
We assume $c\in[0,1)$ from now on.

Since $N$ is odd, the $u$-poles of $F_N(u,\delta)u^{s-1}$ are simple. Thus 
\begin{align*}
\rho_N^{2\pi/N}(k,\delta)=-(2\pi i)^k\Big[&\eta^{-k\lfloor N/4\rfloor }\sum_{\substack{m\in\Z\\ m+c>0}}\frac{(m+c)^{k-1}}{\prod_{\ell=1}^{N-1}\big(1-\e\big(c-(m+c)\eta^\ell\big)\big)}\\
&+\eta^{-k\lfloor 3N/4\rfloor }\sum_{\substack{m\in\Z\\ m+c<0}}\frac{(m+c)^{k-1}}{\prod_{\ell=1}^{N-1}\big(1-\e\big(c-(m+c)\eta^\ell\big)\big)}\Big]
\end{align*} 
for all $k\in\Z$, where $\lfloor \ \rfloor$ denotes the usual floor function. Note that the right-hand side of this equation converges absolutely for all $k\in\Z$. Also, note that $J_g(k,\theta)=1$ if and only if $N$ divides $k$. 

Remarkably, if $N$ divides $k$, we get
\begin{align*}
\sum_{\substack{m\in\Z\\ m+c\not=0}}&\frac{(m+c)^{k-1}}{\prod_{\ell=1}^{N-1}\big(1-\e\big(c-(m+c)\eta^\ell\big)\big)}
=\\
&\begin{cases}
\frac{(-2\pi i)^{1-k}}{(-k)!N}\cdot\LN_N(k,\delta) & \text{if} \ k\leq0,\\
\frac{-(2\pi i)^{1-N}(N-1)!}{N}\cdot\res_{s=N}\big(\LN_N(s,\delta)\big) & \text{if} \ k=N,\\
0 & \text{if} \ k> N.
\end{cases}
\end{align*}
Hence the residue theorem allows us to succinctly write the right-hand side of the last system of equations as 
\begin{align*}
-\frac{(2\pi i)^{1-k}}{N}\cdot \mathrm{coeff}\big(F_N(u,\delta),u^{-k}\big) \qquad\qquad (k\in\Z).
\end{align*}
Such coefficients can be written as $\Q(\eta)$-linear combinations of products of certain generalized Bernoulli polynomials $B_n(z,\xi)$, which are defined via the generating function
\begin{align}\label{Bernoulli.pol.}
\frac{t\e^{zt}}{\xi\e^{t}-1}=\sum_{n=0}^\infty B_n(z,\xi)\frac{t^n}{n!}.
\end{align}
Indeed, for all $k\in\Z$,
\begin{align*}
\mathrm{coeff}\big(F_N(u,\delta),u^{-k}\big)=(-1)^{N-k}\sum_{\substack{n\in\N^N\\ \Tr(n)=N-k}} \eta^{n_1}\eta^{2n_2}\dots\eta^{Nn_N}\prod_{\ell=1}^N\frac{B_{n_\ell}\big(0,\e(c)\big)}{n_\ell!},
\end{align*}
where $\Tr(n):=\Tr(n,P_N)=n_1+\dots+n_N$, and the empty sum is 0. We know that $B_n(0,\e(c))\in\Q(\e(c))$ for all $n\in\N$ (cf. \cite[p. 42, Eq. (3b)]{Hi} and \cite[Eq. (3.7)]{Ap}). Therefore $\mathrm{coeff}(F_N(u,\delta),u^{-k})$ lies in $\Q(\e(c))$ for all $k\in\Z$, since it remains fixed under the action of the Galois group of $\Q(\e(c),\eta)$ over $\Q(\e(c))$. Now we summarize the above discussion.

\begin{corollary}\label{Cor.Ex.1}
Let $N\geq3$ be an odd integer, let $\eta=\e(1/N)$, and let $c\in[0,1)$. Then
\begin{align*}
(2\pi i)^{k-1}\sum_{\substack{m\in\Z\\ m+c\not=0}}\frac{(m+c)^{k-1}}{\prod_{\ell=1}^{N-1}\big(1-\e\big(c-(m+c)\eta^\ell\big)\big)} \
\in \ \Q\big(\e(c)\big)
\end{align*}
for all $k\in N\Z$. Furthermore, assuming $k\in N\Z$, the above series vanishes either for all $k>N$ if $c=0$, or for all $k>0$ if $c\not=0$.
\end{corollary}

\noindent \textbf{Example 2.} Let $N\geq2$, let $\sigma$ be the $N$-cycle $(12\dots N)\in S_N$, let $\eta=\e(1/2N)$, and set $g=T_{\{1\}}R_\sigma M_\eta\in G_N$. Then it follows that $g$ fixes $\delta=(w,a,0)$, where
\begin{align*}
w=(1-\eta)^{-1} \quad\qquad\text{and}\quad\qquad a=(1,\eta,\dots,\eta^{N-1}).
\end{align*}

Again, the $u$-poles of $F_N(u,\delta)u^{s-1}$ are simple, and we get an absolutely convergent series
\begin{align*}
\rho_N^{\pi/N}(k,\delta)=-(2\pi i)^{k}\eta^{(N-k)\lfloor N/2\rfloor}\sum_{m=1}^\infty \frac{\e(-mw)}{\prod_{\ell=1}^{N-1}\big(1-\e(-m\eta^\ell)\big)}m^{k-1} \qquad(k\in\Z).
\end{align*}
In this case, $J_g(k,0)=1$ if and only if $k\equiv N\mod 2N$.

Suppose $k\equiv N\mod 2N$. Using the same reasoning as in the previous example, we obtain
\begin{align*}
\sum_{m=1}^\infty \frac{\e(-mw)}{\prod_{\ell=1}^{N-1}\big(1-\e(-m\eta^\ell)\big)}m^{k-1}=-\frac{(2\pi i)^{1-k}}{2N}\cdot\mathrm{coeff}\big(F_N(u,\delta),u^{-k}\big).
\end{align*}
Then, using the generalized Bernoulli polynomials \eqref{Bernoulli.pol.}, we arrive at the following expression for the above coefficient:
\begin{align*}
(-1)^{N-k}\eta^{(1-N)N/2}\sum_{\substack{n\in\N^N\\ \Tr(n)=N-k}} \eta^{n_2}\eta^{2n_3}\dots\eta^{(N-1)n_N}\prod_{\ell=1}^N\frac{B_{n_\ell}\big(\eta^{1-\ell}w/N,1\big)}{n_\ell!} \quad (k\in\Z).
\end{align*}
In this case, a direct Galois-theoretic approach to study $\mathrm{coeff}(F_N(u,\delta),u^{-k})$ seems awkward, but at least we know that $\mathrm{coeff}(F_N(u,\delta),u^{-k})$ lies in $\Q(\eta)$ for all $k\in\Z$. Now we show that, if $N$ is odd, then $\mathrm{coeff}(F_N(u,\delta),u^{-k})$ is actually a rational number.

Suppose $N$ is odd, and let $\Lambda\in\PP_N$ be the set of odd numbers in $P_N$. Then, by virtue of \eqref{Rho.Trans.Form.}, we have
\begin{align*}
\rho_N^{\pi/N}(k,\delta)&=\eta^{-k}\Big[\rho_N^{2\pi/N}(k,M_{\eta^{-1}}\delta)-\rho_N^{\pi/N}(k,M_{\eta^{-1}}\delta)\Big]\\
&=\eta^{-k}\Big[(-1)^{\frac{N-1}{2}}\rho_N^{2\pi/N}(k,T_\Lambda M_{\eta^{-1}}\delta)+\rho_N^{\pi/N}(k,T_{\{1\}}M_{\eta^{-1}}\delta)\Big].
\end{align*} 
Using basic relations among powers of $\eta$ and suitably permuting coordinates in the rightmost expression, we get
\begin{align*}
\rho_N^{\pi/N}(k,\delta)=\eta^{-k}\Big[(-1)^{\frac{N-1}{2}}\rho_N^{2\pi/N}\big(k,0,(1,\eta^2,\eta^4,\dots,\eta^{2(N-1)}),0\big)+\rho_N^{\pi/N}(k,\delta)\Big],
\end{align*}
from which we easily obtain
\begin{align*}
\rho_N^{2\pi/N}\big(k,0,(1,\eta^2,\eta^4,\dots,\eta^{2(N-1)}),0\big)=(-1)^{\frac{N-1}{2}}(\eta^k-1)\rho_N^{\pi/N}(k,\delta)\qquad(k\in\Z).
\end{align*}
Note that we studied the left-hand side of this equation in the previous example. Hence, using Corollary~\ref{Cor.Ex.1}, it follows that $\rho_N^{\pi/N}(k,\delta)$ lies in $2\pi i\Q$ for all $k\equiv N\mod 2N$. Therefore $\mathrm{coeff}(F_N(u,\delta),u^{-k})$ is rational for all $k\equiv N\mod 2N$.

Now we show another natural application of our results. Since $J_g(0,0)=-1$, Corollary~\ref{Coro.trans.form.fix.points} says that 
$\rho_N^{\pi/N}(0,\delta)=-2\LN_N'(0,\delta)$. In view of Theorem~\ref{Thm.constr.L}~\eqref{Thm.constr.L.rest.}, we want to use the transformations \eqref{T_Lambda} to write $\LN_N'(0,\delta)$ in terms of the Barnes' multiple zeta function. Let $\gamma=(v,b,0)$, where
\begin{align*}
v=\frac{1}{2}\Tr(b,P_N) \qquad\text{and}\qquad b=\big(\eta^{\lfloor N/2\rfloor+1-N},\eta^{\lfloor N/2\rfloor+2-N},\dots,\eta^{\lfloor N/2\rfloor}\big).
\end{align*}
Then, since all the coordinates of $b$ lie in the convex cone $\Co$, we have
\begin{align*}
\rho_N^{\pi/N}(0,\delta)=-2\LN_N'(0,\delta)=2(-1)^{N-\lfloor N/2\rfloor}\LN_N'(0,\gamma)=2(-1)^{N-\lfloor N/2\rfloor}\zeta_N'(0,\gamma).
\end{align*}
On the other hand, using the geometric series and the Taylor expansion of $-\log(1-z)$ at $z=0$, we get the convergent series
\begin{align*}
(-1)^{N-\lfloor N/2\rfloor}&\rho_N^{\pi/N}(0,\delta)=\\
&\sum_{n\in\N^{N-1}}\log\Big(1-\e(-w)\e\big((1+n_1)\eta+\dots+(1+n_{N-1})\eta^{N-1}\big)\Big).
\end{align*}
Taking exponentials and using the notation in \eqref{Mult.Gamma}, we therefore obtain 
\begin{align*}
\Gamma_N(v,b)=\prod_{n\in\N^{N-1}}\Big(1-\e(-w)\e\big((1+n_1)\eta+\dots+(1+n_{N-1})\eta^{N-1}\big)\Big)^{1/2}.
\end{align*}
We now summarize our discussion.

\begin{corollary}\label{Cor.Ex.2}
Let $N\geq2$, let $\eta=\e(1/2N)$, and let $w=(1-\eta)^{-1}$. Then
\begin{align*}
(2\pi i)^{k-1}\sum_{m=1}^\infty \frac{\e(-mw)}{\prod_{\ell=1}^{N-1}\big(1-\e(-m\eta^\ell)\big)}m^{k-1} \ \in \ \begin{cases}\Q(\eta)&  \text{if} \ \text{$N$ is even},\\ \Q& \text{if} \ \text{$N$ is odd},\end{cases}
\end{align*}
for all $k\equiv N\mod 2N$. Furthermore, assuming $k\equiv N\mod 2N$, the above series vanishes for all $k>N$. On the other hand, if $N$ is odd, then the series
\begin{align*}
\sum_{\substack{m\in\Z\\ m\not=0}} \frac{m^{k-1}}{\prod_{\ell=1}^{N-1}\big(1-\e(-m\eta^{2\ell})\big)} 
\end{align*}
vanishes for all $k\in2N\Z$. Also, we have the following evaluation of the Barnes' multiple gamma function in \eqref{Mult.Gamma}:
\begin{align*}
\Gamma_N(v,b)=\prod_{n\in\N^{N-1}}\Big(1-\e(-w)\e\big((1+n_1)\eta+\dots+(1+n_{N-1})\eta^{N-1}\big)\Big)^{1/2},
\end{align*}
where $v=\Tr(b,P_N)/2$ and $b=(\eta^{\lfloor N/2\rfloor+1-N},\eta^{\lfloor N/2\rfloor+2-N},\dots,\eta^{\lfloor N/2\rfloor})$.
\end{corollary}

\section{The group $G$}

Here our main objective is to prove Propositions \ref{prop.R_sigma.T_Lambda} and \ref{prop.R_sigma.T_Lambda.M_alpha}, detailing the many features of the $\mathfrak{T}_N\mathfrak{R}_N$-action on $\Dit_N$ as well as the $G_N$-action on $\D_N$ used in the following. Unless otherwise stated, we employ the notation recorded in the preceding section.

We start by displaying some basic relations satisfied by the trace function $\Tr$ in \eqref{trace} that will prove useful subsequently. We omit the proofs as they come rather easily from its definition.

\begin{lemma}\label{Tr.prop.} Let $N$ be a positive integer.
\begin{enumerate}[$(i)$]
\item \label{Tr.prop.linear} For any $\Lambda\in\PP_N$, the map $v\mapsto\Tr(v,\Lambda)$ defines a linear form $\Tr(\cdot,\Lambda):\C^N\to\C$.
\item \label{Tr.prop.sets} Let $\Lambda_1, \Lambda_2\in\PP_N$. For any $v\in\C^N$, we have 
\begin{align}
&\label{Tr.prop.sets.minus}\Tr(v,\Lambda_1\smallsetminus\Lambda_2)=\Tr(v,\Lambda_1)-\Tr(v,\Lambda_1\cap\Lambda_2),\\
&\label{Tr.prop.sets.union}\Tr(v,\Lambda_1\cup\Lambda_2)=\Tr(v,\Lambda_1)+\Tr(v,\Lambda_2)-\Tr(v,\Lambda_1\cap\Lambda_2).
\end{align}
\item \label{Tr.prop.desc.} Let $\Lambda\in\PP_N$, set $N(\Lambda):=N-|\Lambda|$, and let $\varphi$ be the unique strictly increasing bijection between $P_{N(\Lambda)}$ and $P_N\smallsetminus\Lambda$. For any $\lambda\in\PP_{N(\Lambda)}$ and any $v\in\C^N$, we have $\Tr(v\varphi,\lambda)=\Tr\big(v,\varphi(\lambda)\big)$.
\end{enumerate}
\end{lemma}

\subsection{The group $\mathfrak{M}_N$} For each $\alpha\in\C^*$, recall the function
\begin{align*}
M_\alpha=M_{\alpha,N}:\DD_N\to\DD_N, \qquad\qquad M_\alpha(w,a,\theta):=(\alpha w,\alpha a, \theta),
\end{align*}
defined in \eqref{M_alpha}. We denote $\mathfrak{M}_N$ the group generated by the $M_\alpha$ ($\alpha\in\C^*$), which is naturally isomorphic to $\C^*$. Let $\mathrm{C}(\mathfrak{M}_N)$ and $\mathrm{N}(\mathfrak{M}_N)$ be the centralizer and normalizer of $\mathfrak{M}_N$ in $\Aut(\DD_N)$ respectively, so $\mathrm{C}(\mathfrak{M}_N)\subseteq \mathrm{N}(\mathfrak{M}_N)$.

\begin{lemma}\label{lemma.M_alpha}
The group $\mathfrak{M}_N$ acts on the set $\E$ described in Definition~\ref{Epsilon} by 
\begin{align*}
\epsilon\cdot M_\alpha=\epsilon_\alpha:\DD_N\to\R_+, \quad \epsilon_\alpha(w,a,\theta):=|\alpha|\cdot\epsilon(\alpha w,\alpha a,\theta) \quad\quad (\alpha\in\C^*,\, \epsilon\in\E).
\end{align*}
Furthermore,
\begin{align*}
F_N\big(u,M_\alpha(w,a,\theta)\big)=F_N(\alpha u, w,a,\theta) \qquad\qquad (\alpha\in\C^*)
\end{align*}
for all $(w,a,\theta)\in\DD_N$ and all $u\in\C$ such that both $u$ and $\alpha u$ are away from singularities.
\end{lemma}

\begin{proof}
We first claim that $\epsilon_\alpha\in\E$ for all $\epsilon\in\E$ and all $\alpha\in\C^*$. Indeed, let $\epsilon\in\E$, $\alpha\in\C^*$, $(w,a,\theta)\in\DD_N$, and $\ell\in P_N$. Since $\epsilon\in\E$ and $M_\alpha(w,a,\theta)\in\DD_N$, we know by definition that the equation $\e(\theta_\ell)=\e^{u(\alpha a_\ell)}$ has no solutions $u\in\C$ with $0<|u|\leq \epsilon(\alpha w,\alpha a,\theta)$. But $0<|u|\leq\epsilon(\alpha w,\alpha a,\theta)$ if and only if $0<|\alpha u|\leq\epsilon_\alpha(w,a,\theta)$, so the claim follows after applying the change of variable $\alpha u=u'$. Then it is easy to prove that this defines a right action of $\mathfrak{M}_N$ on $\E$. The last assertion follows from the definition \eqref{F_N} of $F_N$ by a direct computation.
\end{proof}

\subsection{The group $\mathfrak{R}_N$} For each $\sigma\in S_N$, recall the function 
\begin{align*}
R_\sigma=R_{\sigma,N}:\DD_N\to\DD_N, \qquad\qquad R_\sigma(w,a,\theta):=\big(w \, , \, a r(\sigma) \, , \, \theta r(\sigma)\big),
\end{align*}
defined in \eqref{R_sigma}. Here $r(\sigma)$ denotes the $N\times N$ square matrix whose $\ell$-th column equals the $\sigma^{-1}(\ell)$-th column of the identity matrix of size $N$. It is well-known that $r(\sigma)r(\tau)=r(\tau\sigma)$ for all $\sigma, \tau\in S_N$. Thus we have that $R_\tau R_\sigma=R_{\tau\sigma}$ for all $\sigma, \tau\in S_N$, and that the group $\mathfrak{R}_N$ generated by the $R_\sigma$ is isomorphic to $S_N$. 

\begin{lemma}\label{lemma.R_sigma}
The group $\mathfrak{R}_N$ is contained in $\Aut(\DD_N)_\pi\cap \mathrm{C}(\mathfrak{M}_N)$, and it acts on the set $\E$ described in Definition~\ref{Epsilon} by composition on the right. Furthermore,
\begin{align*}
F_N\big(u,R_\sigma(w,a,\theta)\big)=F_N(u,w,a,\theta) \qquad\qquad (\sigma\in S_N)
\end{align*}
for all $(w,a,\theta)\in\DD_N$ and all $u\in\C$ away from singularities.
\end{lemma}

\begin{proof}
Let $(w,a,\theta)\in\DD_N$. We obviously have 
\begin{align*}
\Tr(a,a^{-1}[-\Co])=\Tr\big(ar(\sigma)\,,\,(ar(\sigma))^{-1}[-\Co]\big) \qquad\qquad (\sigma\in S_N),
\end{align*}
which implies $\pi R_\sigma=\pi$ for all $\sigma\in S_N$. On the other hand, it readily follows that $R_\sigma M_\alpha=M_\alpha R_\sigma$ for all $\sigma\in S_N$ and all $\alpha\in\C^*$. Therefore $\mathfrak{R}_N$ is contained in $\Aut(\DD_N)_\pi\cap \mathrm{C}(\mathfrak{M}_N)$.

Since every $R_\sigma$ acts on an element $(w,a,\theta)\in\DD_N$ by permuting simultaneously the entries of both $a$ and $\theta$ according to $\sigma$, we can see in Definition~\ref{Epsilon} that the composition $\epsilon R_\sigma\in\E$ for all $\sigma\in S_N$. Then it easily follows that this defines an action of $\mathfrak{R}_N$ on $\E$.

Finally, we can see directly from \eqref{F_N} that the test function $F_N$ is invariant under the action (by homeomorphisms) of $\mathfrak{R}_N$ on $\DD_N$.
\end{proof}

\subsection{The group $\mathfrak{T}_N$}

For each $\Lambda\in\PP_N$, recall the function 
\begin{align*}
T_\Lambda=T_{\Lambda,N}:\DD_N\to \DD_N, \qquad
  T_\Lambda(w,a,\theta):=\big(w-\Tr(a,\Lambda) \, , \,  ad(\Lambda) \, , \, \theta d(\Lambda)\big),
\end{align*}
defined in \eqref{T_Lambda}, where $d(\Lambda)$ denotes the $N\times N$ diagonal matrix whose $(\ell,\ell)$-entry equals either $-1$ if $\ell\in\Lambda$, or $1$ otherwise. We define $\mathfrak{T}_N:=\{T_\Lambda\,|\,\Lambda\in\PP_N\}$.

\begin{lemma}\label{lemma.T_Lambda}
The set $\mathfrak{T}_N$ is a subgroup of $\Aut(\DD_N)_\pi\cap \mathrm{C}(\mathfrak{M}_N)$ isomorphic to $\{\pm1\}^N$, and its group operation satisfies $T_{\Lambda_1} T_{\Lambda_2}=T_{\Lambda_1\oplus\Lambda_2}$ $(\Lambda_1,\,\Lambda_2\in\PP_N)$. Furthermore, $\mathfrak{T}_N$ acts on the set $\E$ described in Definition~\ref{Epsilon} by composition on the right, and we have
\begin{align*}
F_N\big(u,T_\Lambda(w,a,\theta)\big)=(-1)^{|\Lambda|}\e\big(\Tr(\theta,\Lambda)\big)F_N(u,w,a,\theta)
\end{align*}
for all $(w,a,\theta)\in\DD_N$ and all $u\in\C$ away from singularities.
\end{lemma}

\begin{proof}
We first prove the composition law. Let $a\in\C^N$ and let $\Lambda_1, \Lambda_2\in\PP_N$. Using \eqref{Tr.prop.sets.minus}, Lemma~\ref{Tr.prop.}~\eqref{Tr.prop.linear}, and \eqref{Tr.prop.sets.union}, it verifies that
\begin{align*}
&\Tr(a,\Lambda_2)+\Tr\big(ad(\Lambda_2)\,,\,\Lambda_1\big)\\
&=\Tr(a,\Lambda_2\smallsetminus\Lambda_1)+\Tr(a,\Lambda_1\cap\Lambda_2)+\Tr\big(ad(\Lambda_2)\,,\,\Lambda_1\smallsetminus\Lambda_2\big)+\Tr\big(ad(\Lambda_2)\,,\,\Lambda_1\cap\Lambda_2\big)\\
&=\Tr(a,\Lambda_2\smallsetminus\Lambda_1)+\Tr(a,\Lambda_1\cap\Lambda_2)+\Tr(a,\Lambda_1\smallsetminus\Lambda_2)+\Tr(-a,\Lambda_1\cap\Lambda_2)\\
&=\Tr(a,\Lambda_2\smallsetminus\Lambda_1)+\Tr\big(a\,,\,\Lambda_1\smallsetminus\Lambda_2\big)\\
&=\Tr(a,\Lambda_1\oplus\Lambda_2).
\end{align*}
On the other hand, it can be easily proved that $d(\Lambda_2)d(\Lambda_1)=d(\Lambda_1\oplus\Lambda_2)$. Hence the identity $T_{\Lambda_1}T_{\Lambda_2}=T_{\Lambda_1\oplus\Lambda_2}$ follows immediately.

Now we prove that $\mathfrak{T}_N$ is a subgroup of $\Aut(\DD_N)_\pi\cap\mathrm{C}(\mathfrak{M}_N)$. First note that, for every $\Lambda\in\PP_N$, $T_\Lambda$ is a continuous function whose square equals the identity map on $\DD_N$. Then all the $T_\Lambda$ are homeomorphisms. Next observe that, for every $a\in\C^{N}$ and every $\Lambda\in\PP_N$,
\begin{align*}
\big(ad(\Lambda)\big)^{-1}[-\Co]\,=\,\big(a^{-1}[-\Co]\smallsetminus\Lambda\big)\,\cup\, \big(a^{-1}[\Co]\cap\Lambda\big) \qquad\qquad (\text{disjoint union}).
\end{align*}
Then, using \eqref{Tr.prop.sets.minus}, Lemma~\ref{Tr.prop.}~\eqref{Tr.prop.linear}, and \eqref{Tr.prop.sets.union}, we get 
\begin{align*}
&\Tr(a,\Lambda)+\Tr\big(ad(\Lambda)\,,\,(ad(\Lambda))^{-1}[-\Co]\big)\\
&=\Tr(a,\Lambda\smallsetminus a^{-1}[\Co])+\Tr(a,a^{-1}[\Co]\cap\Lambda)+\Tr(a,a^{-1}[-\Co]\smallsetminus\Lambda)+\Tr\big(-a,a^{-1}[\Co]\cap\Lambda\big)\\
&=\Tr(a,\Lambda\smallsetminus a^{-1}[\Co])+\Tr(a,a^{-1}[-\Co]\smallsetminus\Lambda)\\
&=\Tr(a,\Lambda\cap a^{-1}[-\Co])+\Tr(a,a^{-1}[-\Co]\smallsetminus\Lambda)\\
&=\Tr(a,a^{-1}[-\Co]).
\end{align*} 
Hence we easily obtain $\pi T_\Lambda=\pi$ for all $\Lambda\in\PP_N$. On the other hand, using the linearity of $\Tr$ in the first variable, it is straightforward to verify that $M_\alpha T_\Lambda=T_\Lambda M_\alpha$ for all $\alpha\in\C^*$ and all $\Lambda\in\PP_N$. Therefore the above shows that $\mathfrak{T}_N$ is a subgroup of $\Aut(\DD_N)_\pi\cap\mathrm{C}(\mathfrak{M}_N)$. In order to prove that $\mathfrak{T}_N$ is isomorphic to $\{\pm1\}^N$, it suffices to consider the mapping $T_\Lambda\mapsto (1,1,\dots,1)d(\Lambda)$.

Now we prove that $\mathfrak{T}_N$ acts on $\E$ by composition on the right. We first show that $\epsilon T_\Lambda\in\E$ for all $\Lambda\in\PP_N$ and all $\epsilon\in\E$. Let $\epsilon\in\E$, $\Lambda\in\PP_N$, $(w,a,\theta)\in\DD_N$, and $\ell\in P_N$. Since $\epsilon\in\E$, the equation $\e\big((\theta d(\Lambda))_\ell\big)=\e^{u(ad(\Lambda))_\ell}$ has no solutions $u\in\C$ with $0<|u|\leq \epsilon T_\Lambda(w,a,\theta)$. But $\e\big((\theta d(\Lambda))_\ell\big)=\e^{u(ad(\Lambda))_\ell}$ if and only if $\e(\theta_\ell)=\e^{ua_\ell}$, so $\epsilon T_\Lambda\in\E$. Then it readily verifies that the above defines an action of $\mathfrak{T}_N$ on $\E$. Finally, the last assertion follows from the definition \eqref{F_N} of $F_N$ by a direct computation.
\end{proof}

\begin{proof}[Proof of Propositions 1 and 2]
By virtue of Lemmas \ref{lemma.R_sigma} and \ref{lemma.T_Lambda}, it only remains to show that $R_\sigma T_\Lambda R_\sigma^{-1}=T_{\sigma(\Lambda)}$ for all $\Lambda\in\PP_N$ and all $\sigma\in S_N$. 

For any set $X$ and any subset $S\subseteq X$, let 
\begin{align}\label{indicator.funct.}
\ind_S:X\to\{0,1\}, \qquad\qquad \ind_S(x):=\begin{cases}1 \quad \text{if} \ x\in S,\\ 0 \quad \text{if} \ x\notin S,\end{cases}
\end{align}
be the usual indicator function. Let $\Lambda\in\PP_N$ and $\sigma\in S_N$. Since
\begin{align*}
ar(\sigma^{-1})=(a_{\sigma(1)},\dots,a_{\sigma(N)}) \qquad\text{and}\qquad ad(\Lambda)=\big((-1)^{\ind_\Lambda(1)}a_1,\dots,(-1)^{\ind_\Lambda(N)}a_N\big)
\end{align*}
for any $a\in\C^N$, we have 
\begin{align}\label{Lambda&Sigma.relations}
\Tr(ar(\sigma^{-1}),\Lambda)=\Tr(a,\sigma(\Lambda)) \qquad\text{and}\qquad r(\sigma^{-1})d(\Lambda)r(\sigma)=d(\sigma(\Lambda)).
\end{align}
Then the desired identity follows easily by using \eqref{T_Lambda} and \eqref{R_sigma}.
\end{proof}

\section{The integral representation of $\LN$}

In this section our main objective is twofold: to prove Proposition~\ref{prop.fixed.points}, and to establish the integral representation of $\LN$. In order to address the latter part, we start by defining a set of parameters to give a meaning to the formal definition of $\LN$ (see Definition~\ref{Epsilon}).

For any positive integer $N$, let
\begin{align}\label{Di_N}
\Di_N:=\big\{(w,a,\theta)\in\DD_N \, \big| \,  \pi(w,a,\theta)\in\Co^\circ \quad \text{and} \quad a\in[\C\smallsetminus i\cdot\R^*]^N\big\},
\end{align}
where $\R^*:=\R\smallsetminus\{0\}$, $\pi$ as in \eqref{pi}, and $\DD_N$ as in \eqref{good.domain}. To consider also the case $N=0$, let
\begin{align}
& P_0:=\emptyset, \quad \PP_0:=\{\emptyset\}, \quad \DD_0:=\C,  \quad \text{and} \quad \Di_0=\mathcal{T}_0^+:=\{w\in\C \ | \ \re(w)>0\},
   \label{D_0}
\end{align}
where $\emptyset$ denotes the empty set.

\subsection{A family of projections}  

In order to begin with the study of $\Di_N$, we introduce a family of projections where the distinguished element $\pi$ fits in naturally.

\begin{definition}\label{projectiondefinition}
Let $N\in\N$ and $\Lambda\in \PP_N$. Set $N(\Lambda):=N-|\Lambda|$.
\begin{enumerate}[$(i)$]
\item If $\Lambda=\emptyset$, we define $\pi_\Lambda=\pi_{\Lambda,N}:\DD_N\to\DD_{N(\Lambda)}$ to be the identity map.
\item If $\Lambda=P_N\not=\emptyset$, we define $\pi_\Lambda=\pi_{\Lambda,N}:\DD_N\to\DD_{N(\Lambda)}$ by 
\begin{align*}
\pi_\Lambda(w,a,\theta):=w-\Tr(a,a^{-1}[-\Co]), \qquad\qquad (w,a,\theta)\in \DD_N.
\end{align*}
\item If $\Lambda$ is a non-empty proper subset of $P_N$, we define $\pi_\Lambda=\pi_{\Lambda,N}:\DD_N\to\DD_{N(\Lambda)}$ by 
\begin{align*}
\pi_\Lambda(w,a,\theta):=\big(w-\Tr(a,\Lambda\cap a^{-1}[-\Co]) \, ,\, \widehat{a}_\Lambda \, , \, \widehat{\theta}_\Lambda\big), \qquad\qquad (w,a,\theta)\in \DD_N.
\end{align*}
Here, for every $v\in\C^{N}$, we write $\widehat{v}_\Lambda\in \C^{N(\Lambda)}$ for the composition $v\varphi:P_{N(\Lambda)}\to\C$, where $\varphi$ is the unique strictly increasing bijection between $P_{N(\Lambda)}$ and $P_N\smallsetminus\Lambda$.
\end{enumerate}
\end{definition}

We remark that $\pi$ corresponds to $\pi_{P_N}$ in this setting. However, we will usually keep on writing $\pi$ in place of $\pi_{P_N}$ for the sake of simplicity.

\begin{lemma}\label{bef.projection.lemma}
Let $N\in\N$ and $\Lambda\in\PP_N$. 
\begin{enumerate}[$(i)$]
\item \label{bef.projection.lemma.limit} If $\{(W_n,A_n,\Theta_n)\}_{n\in\N}$ is a sequence in $\DD_N$ converging to $(w,a,\theta)\in\Di_N$, then
\begin{align*}
\lim_{n\to\infty}\pi_\Lambda(W_n,A_n,\Theta_n)=\pi_\Lambda(w,a,\theta).
\end{align*}
\item \label{bef.projection.lemma.Di} The set $\Di_N$ is path-connected and contains $\mathcal{T}_N^+$. Furthermore, if $N\geq1$, its interior
\begin{align}
\label{interiorDi} \Dio_N=\big\{(w,a,\theta)\in\Di_N \, \big| \,  a^{-1}[0]=\emptyset\big\}
\end{align}
can be decomposed into $2^N$ simply-connected components $\Dio_{N,\Lambda}$ $(\Lambda\in\PP_N)$, where
\begin{align*}
\Dio_{N,\Lambda}:=\big\{(w,a,\theta)\in\Dio_N\,\big|\,a^{-1}[-\Co]=\Lambda\big\}.
\end{align*}
\end{enumerate}
\end{lemma}

\begin{proof}
We assume that $\Lambda$ is non-empty (so $N\geq1$) since otherwise both (i) and (ii) follow trivially. We first prove (i). Since $\{A_n\}_{n\in\N}$ converges to $a$ and $\re(a_\ell)=0$ only if $a_\ell=0$ ($\ell\in P_N$), we have 
\begin{align*}
a^{-1}[-\Co]=a^{-1}[-\Co^\circ]=A_n^{-1}[-\Co]\smallsetminus a^{-1}[0] \qquad\qquad \text{(for all big enough $n$).}
\end{align*}
Then, using \eqref{Tr.prop.sets.minus}, it follows that
\begin{align*}
\Tr\big(A_n\,,\,\Lambda\cap A_n^{-1}[-\Co]\big)=\Tr\big(A_n\,,\,\Lambda\cap a^{-1}[-\Co]\big)+\Tr\big(A_n\,,\,\Lambda\cap A_n^{-1}[-\Co]\cap a^{-1}[0]\big)
\end{align*}
for all big enough $n$. Since $\Tr$ is continuous in the first variable (see Lemma~\ref{Tr.prop.}~\eqref{Tr.prop.linear}), the first term in the right-hand side of the last equation converges to $\Tr\big(a,\Lambda\cap a^{-1}[-\Co]\big)$ as $n$ approaches infinity. On the other hand, we have
\begin{align*}
\lim_{n\to\infty}\big|\Tr\big(A_n\,,\,\Lambda\cap A_n^{-1}[-\Co]\cap a^{-1}[0]\big)\big|\leq\lim_{n\to\infty}\sum_{\ell\in a^{-1}[0]}|A_{n,\ell}|=0.
\end{align*}
Hence $\Tr(A_n,\Lambda\cap A_n^{-1}[-\Co])$ converges to $\Tr(a,\Lambda\cap a^{-1}[-\Co])$ as $n$ approaches infinity. Then (i) follows directly from the definition of $\pi_\Lambda$.

Now we prove (ii). Recall that $(w,a,\theta)\in \mathcal{T}_N^+$ only if $\re(w)$ and $\re(a_\ell)$ are positive for all $\ell\in P_N$. Then, for all $(w,a,\theta)\in \mathcal{T}_N^+$, we have $\pi(w,a,\theta)=w\in\Co^\circ$ and $a\in[\Co^\circ]^N$ since $\Co^\circ$ is nothing but the right half-plane. This proves $\mathcal{T}_N^+\subseteq \Di_N$.

Let $\delta=(w,a,\theta)$ be an arbitrary element in $\Di_N$, and let $\delta_0=(w_0,a_0,\theta_0)$ be the element with $w_0=1$, $a_{0\ell}=0$, and $\theta_{0\ell}=1/2$ for all $\ell\in P_N$. We claim that there exists a path in $\Di_N$ connecting $\delta$ and $\delta_0$. Indeed, we construct such a path by composing at most two line segments.
First consider the one joining $\delta$ and $\delta_1=(w_1,a_1,\theta_1)$, where $w_1=1$, $a_{1\ell}=\ind_{a^{-1}[0]}(\ell)$, and $\theta_{1\ell}=1/2$ for all $\ell\in P_N$. Here $\ind_{a^{-1}[0]}$ is the indicator function defined in \eqref{indicator.funct.}. Then consider the line segment joining $\delta_1$ and $\delta_0$. It can be easily shown that both lie entirely in $\Di_N$ proving thus our claim. Since $\delta$ is arbitrary, this implies that $\Di_N$ is path-connected.

To prove \eqref{interiorDi}, first note that
\begin{align*}
\big\{(w,a,\theta)\in\Di_N \, | \, a^{-1}[0]=\emptyset \big\}
&=\pi^{-1}[\Co^\circ]\cap \big(\C\times[\C\smallsetminus i\cdot\R]^N\times\R^N\big),
\end{align*}
where the set in the right-hand side is open since the restriction of $\pi$ to $\Di_N$ is continuous by virtue of (i). Then 
\begin{align*}
\{(w,a,\theta)\in\Di_N \, | \, a^{-1}[0]=\emptyset \}\subseteq \Dio_N.
\end{align*}
To prove the reverse inclusion, take $\delta=(w,a,\theta)\in\Di_N$ such that $a^{-1}[0]\not=\emptyset$. Then there exists $\ell\in P_N$ such that $a_\ell=0$, and so any neighborhood of $a_\ell$ will contain nonzero purely imaginary complex numbers. This implies that any neighborhood of $\delta$ will contain elements lying outside $\Di_N$. This proves \eqref{interiorDi}.

To prove the last statement of \eqref{bef.projection.lemma.Di}, note that we can easily write $\Dio_N$ as the disjoint union of the $\Dio_{N,\Lambda}$ ($\Lambda\in \PP_N$). Also, note that the map $v\mapsto v^{-1}[-\Co]$ defines a locally constant function from $[\C\smallsetminus i\cdot\R]^N$ to $\PP_N$, which implies that every $\Dio_{N,\Lambda}$ is open. Finally, we prove that each $\Dio_{N,\Lambda}$ is convex (and so simply-connected). Let $\Lambda\in\PP_N$ and define
\begin{align*}
\DD_{N,\Lambda}:=\big\{(w,a,\theta)\in\DD_N\,\big|\,a^{-1}[-\Co]=\Lambda\big\}.
\end{align*}
For any pair of elements $\delta_0=(w_0,a_0,\theta_0)$ and $\delta_1=(w_1,a_1,\theta_1)$ of $\Dio_{N,\Lambda}$, consider $\delta_t:=(1-t)\delta_0+t\delta_1$ for every $t\in[0,1]$. Since both $a_0$ and $a_1$ belong to $[\C\smallsetminus i\cdot\R]^N$ and satisfy $a_0^{-1}[-\Co]=a_1^{-1}[-\Co]=\Lambda$, it readily follows that $\delta_t$ is contained in $\DD_{N,\Lambda}$ for all $t\in[0,1]$. On the other hand, the restriction of $\pi$ to $\DD_{N,\Lambda}$ is given by $\pi(w,a,\theta)=w-\Tr(a,\Lambda)$, so $\pi(\delta_t)=(1-t)\pi(\delta_0)+t\pi(\delta_1)$ for all $t\in[0,1]$. Since both $\pi(\delta_0)$ and $\pi(\delta_1)$ lie in the convex set $\Co^\circ$, we obtain that $\pi(\delta_t)\in\Co^\circ$ for all $t\in[0,1]$.  This proves that the line segment joining $\delta_0$ and $\delta_1$ lies entirely in $\Dio_{N,\Lambda}$. 
\end{proof}

\begin{proposition}\label{projection.lemma}
Let $N\in\N$ and $\Lambda\in\PP_N$. Set $N(\Lambda):=N-|\Lambda|$.
\begin{enumerate}[$(i)$]
\item For any $\lambda\in\PP_{N(\Lambda)}$, we have $\pi_\lambda\pi_\Lambda=\pi_{\varphi(\lambda)\cup\Lambda}$,
where $\varphi$ is the unique strictly increasing bijection between $P_{N(\Lambda)}$ and $P_N\smallsetminus\Lambda$.
\item The function $\pi_\Lambda$ is surjective, and its restriction to $\Di_N$ induces a continuous surjective map $\pi_\Lambda:\Di_N\to\Di_{N(\Lambda)}$ that sends $\{(w,a,\theta)\in\Di_N \, | \, a^{-1}[0]=\Lambda\}$ onto $\Dio_{N(\Lambda)}$.
\item \label{projection.lemma.Re[pi]} The function $\re[\pi]:\DD_N\to\R$ given by $\re[\pi](w,a,\theta):=\re(\pi(w,a,\theta))$ is continuous and surjective.
\end{enumerate}
\end{proposition}

\begin{proof}
To prove (i), we will make use of the identity
\begin{align}\label{varphi.prop.}
\varphi\big((v\varphi)^{-1}[S]\big)=v^{-1}[S]\smallsetminus\Lambda \qquad\qquad (v\in\C^N,\, S\subseteq\C).
\end{align}  
We assume that $\Lambda$ is a non-empty proper subset of $P_N$ and that $\lambda$ is non-empty since otherwise the equality $\pi_\lambda\pi_\Lambda=\pi_{\varphi(\lambda)\cup\Lambda}$ holds trivially. Note that 
\begin{align*}
&\Tr\big(a\,,\,\Lambda\cap a^{-1}[-\Co]\big)+\Tr\big(a\varphi\,,\,\lambda\cap(a\varphi)^{-1}[-\Co]\big)&  \\
&=\Tr\big(a\,,\,\Lambda\cap a^{-1}[-\Co]\big)+\Tr\big(a\,,\,\varphi(\lambda)\cap (a^{-1}[-\Co]\smallsetminus\Lambda)\big) & [\text{Lemma~\ref{Tr.prop.}~\eqref{Tr.prop.desc.} and \eqref{varphi.prop.}}]\\
&= \Tr\big(a\,,\,(\varphi(\lambda)\cup\Lambda)\cap a^{-1}[-\Co]\big) & [\text{by using \eqref{Tr.prop.sets.minus} repeatedly}] 
\end{align*}
for all $(w,a,\theta)\in\DD_N$. On the other hand, if $\lambda\not=P_{N(\Lambda)}$, let $N(\Lambda,\lambda):=N-|\Lambda|-|\lambda|$, and let $\overline{\varphi}$ be the strictly increasing bijection from $P_{N(\Lambda,\lambda)}$ onto $P_{N(\Lambda)}\smallsetminus\lambda$. Then it verifies that $\varphi\overline{\varphi}$ is the strictly increasing bijection from $P_{N(\Lambda,\lambda)}$ onto $P_N\smallsetminus\big(\varphi(\lambda)\cup\Lambda\big)$. Consequently, (i) follows from Definition~\ref{projectiondefinition}.

Now we prove (ii). We assume that $\Lambda$ is non-empty since otherwise the assertion holds trivially. Note that this assumption implies $N\geq1$. We first show that both $\pi_\Lambda$ and its restriction to $\Di_N$ are onto.

Suppose that $\Lambda=P_N$. For each $w\in\DD_0$, take the element $(w,a,\theta)\in\DD_N$ such that $a_\ell=0$ and $\theta_\ell=1/2$ for all $\ell\in P_N$. Then $\pi(w,a,\theta)=w$ by \eqref{pi}. Finally, in view of \eqref{Di_N} and \eqref{D_0}, we have that $(w,a,\theta)\in\Di_N$ if and only if $w\in\Di_0$.

Suppose that $\Lambda\not=P_N$. For each $v\in\C^{N(\Lambda)}$ and each $\alpha\in\C$, we define $v_{\Lambda}[\alpha]$ to be the unique element in $\C^{N}$ such that
\begin{itemize}
\item the coordinates of $v_{\Lambda}[\alpha]$ indexed by elements of $\Lambda$ are all equal to $\alpha$, and
\item the projection of $v_{\Lambda}[\alpha]$ onto the coordinates indexed by elements of $P_N\smallsetminus \Lambda$ is $v$.
\end{itemize}
Now, for each $(w,a,\theta)\in\DD_{N(\Lambda)}$, take the element $(w,a_{\Lambda}[0],\theta_{\Lambda}[1/2])\in\DD_N$. It follows that $\pi_{\Lambda}(w,a_{\Lambda}[0],\theta_{\Lambda}[1/2])=(w,a,\theta)$. Then, from (i) and \eqref{Di_N}, we get that $(w,a_{\Lambda}[0],\theta_{\Lambda}[1/2])\in\Di_N$ if and only if $(w,a,\theta)\in\Di_{N(\Lambda)}$.

It is clear that Lemma~\ref{bef.projection.lemma}~\eqref{bef.projection.lemma.limit} implies that $\pi_\Lambda:\Di_N\to\Di_{N(\Lambda)}$ is continuous. The last statement in (ii) is a consequence of the description of $\Dio_{N(\Lambda)}$ given in \eqref{interiorDi} and the definition of $\pi_\Lambda$.

Finally, to prove (iii), first note that $\re[\pi]$ is onto since it is the composition of two surjective functions. To prove continuity, let $\{(W_n,A_n,\Theta_n)\}_{n\in\N}$ be a sequence in $\DD_N$ converging to $(w,a,\theta)\in\DD_N$. Then we have
\begin{align*}
a^{-1}[-\Co^\circ]=A_n^{-1}[-\Co]\smallsetminus a^{-1}[i\cdot\R] \qquad\qquad \text{(for all big enough $n$).}
\end{align*}
Thus \eqref{Tr.prop.sets.minus} implies that
\begin{align*}
\Tr\big(A_n\,,\,A_n^{-1}[-\Co]\big)=\Tr\big(A_n\,,\,a^{-1}[-\Co^\circ]\big)+\Tr\big(A_n\,,\,A_n^{-1}[-\Co]\cap a^{-1}[i\cdot\R]\big)
\end{align*}
for all big enough $n$. Since $\Tr$ is continuous in the first variable, the first term in the right-hand side of the last equation converges to $\Tr(a,a^{-1}[-\Co^\circ])$ as $n\to+\infty$. For the second term we have
\begin{align*}
\lim_{n\to\infty}\Big|\re\Big(\Tr\big(A_n\,,\,A_n^{-1}[-\Co]\cap a^{-1}[i\cdot\R]\big)\Big)\Big|\leq\lim_{n\to\infty}\sum_{\ell\in a^{-1}[i\cdot\R]}\big|\re(A_{n,\ell})\big|=0.
\end{align*}
Hence 
\begin{align*}
\lim_{n\to\infty}\re\big(\Tr\big(A_n\,,\,A_n^{-1}[-\Co]\big)\big)=\re\big(\Tr\big(a,a^{-1}[-\Co^\circ]\big)\big)=\re\big(\Tr\big(a,a^{-1}[-\Co]\big)\big),
\end{align*}
and so $\re[\pi](W_n,A_n,\Theta_n)$ converges to $\re[\pi](w,a,\theta)$ as $n\to+\infty$, which proves continuity.
\end{proof}

\subsection{The integral representation of $\LN$} Now we establish the integral representation of $\LN_N$. First, in order to treat the case $N=0$, we set
\begin{align*}
F_0(u,w):=\e^{-uw} \qquad\qquad (u,w\in\C).
\end{align*}
Then an elementary computation shows that 
\begin{align*}
w^{-s}=\frac{1}{\Gamma(s)(\e(s)-1)}\int_{C(\epsilon)}F_0(u,w)u^{s-1}du \qquad\qquad (w\in\Di_0, \ s\in\C)
\end{align*}
for every small enough $\epsilon>0$, where the complex power $w^{-s}$ is defined using the principal branch of the logarithm. Consequently we define 
\begin{align*}
\LN_0:\C\times \Di_0\to\C, \qquad\qquad
\LN_0(s,w):=w^{-s},
\end{align*}
which is clearly a holomorphic function on $\C\times \Di_0$. 

\begin{proposition}\label{LN}
Let $N$ be a positive integer. The function $\LN_N:(\C\smallsetminus P_N)\times \Di_N\to\C$, 
\begin{align*}
\LN_N(s,w,a,\theta):=\frac{1}{\Gamma(s)(\e(s)-1)}\int_{C(\epsilon(w,a,\theta))}F_N(u,w,a,\theta)u^{s-1}du,
\end{align*}
is well-defined and independent of the choice of $\epsilon\in\E$. Furthermore, it satisfies the following properties.
\begin{enumerate}[$(i)$]
\item\label{LN.merom.} For any $(w,a,\theta)\in \Di_N$, the map $s\mapsto \LN_N(s,w,a,\theta)$ defines a meromorphic function on $\C$ having at most simple poles at $P_N$.
\item\label{LN.restr.} For any $(s,w,a,\theta)\in(\C\smallsetminus P_N)\times \mathcal{T}_N^+$, we have $\LN_N(s,w,a,\theta)=\zeta_N(s,w,a,\theta)$.
\item\label{LN.proj.} For any $(s,w,a,\theta)\in(\C\smallsetminus P_N)\times\Di_N$, we have
\begin{align*}
\LN_N(s,w,a,\theta)=\LN_{N(a^{-1}[0])}\big(s,\pi_{a^{-1}[0]}(w,a,\theta)\big)\cdot\prod_{\ell\in a^{-1}[0]}\big(1-\e(\theta_\ell)\big)^{-1}.
\end{align*}
\end{enumerate}
\end{proposition} 

\begin{proof}
Take $\epsilon\in\E$ and $(w,a,\theta)\in\Di_N$. We claim that the integral 
\begin{align}\label{Integral}
\IN_N(s,w,a,\theta):=\int_{C(\epsilon(w,a,\theta))}F_N(u,w,a,\theta)u^{s-1}du \qquad\qquad (s\in\C)
\end{align}
defines a holomorphic function of $s$ on $\C$ independent of the choice of $\epsilon$ in $\E$. Indeed, conditions $a\in[\C\smallsetminus i\cdot\R^*]^N$ (see \eqref{Di_N}) and $a^{-1}[0]\subseteq\theta^{-1}[\R\smallsetminus\Z]$ (see \eqref{good.domain}), together with Definition~\ref{Epsilon}, imply that $F_N(u,w,a,\theta)$ is a regular function of $u$ on $C(\epsilon(w,a,\theta))$. Condition $\pi(w,a,\theta)\in\Co^\circ$ implies that the absolute value of $F_N(u,w,a,\theta)$ decays exponentially as $u$ goes to $+\infty$ along the real line. Hence $\IN_N(s,w,a,\theta)$ is a holomorphic function of $s$ on $\C$. To prove that it is independent of the choice of $\epsilon$ in $\E$, first note that $u=0$ is the only possible singularity of $F_N(u,w,a,\theta)$ having absolute value $\leq\epsilon'(w,a,\theta)$, for all $\epsilon'\in \E$. Then we can use Cauchy's integral formula to finish the proof of our claim. 

The above paragraph says that the singularities of $s\mapsto\LN_N(s,w,a,\theta)$ come from the factor $\Gamma(s)^{-1}(\e(s)-1)^{-1}$. Since the residue theorem implies that $\IN_N(k,w,a,\theta)=0$ for every integer $k\geq N+1$, we have that those singularities are at most simple poles lying in $P_N$. This proves (i). 

Assertion (ii) follows from \eqref{anal.cont.zeta} and Lemma~\ref{bef.projection.lemma}~\eqref{bef.projection.lemma.Di}, while (iii) follows readily from the definitions of $F_N$ (see \eqref{F_N}) and $\LN_N$.
\end{proof}

\subsection{Transformation formula} In this subsection we prove that $\Di_N$ is invariant under $\mathfrak{T}_N\mathfrak{R}_N$-transformations, and we provide the corresponding transformation formula for $\LN_N$. In addition, even though the group $\C^*$ does not act on $\Di_N$, as it is clear from definition \eqref{Di_N}, we derive a transformation formula for $\LN_N$ under $\C^*$-transformations from assuming certain conditions.

\begin{proposition}\label{T_NR_N.trans.form.}
Let $N$ be a positive integer. 
\begin{enumerate}[$(i)$] 
\item\label{T_NR_N.trans.form.act.} The group $\mathfrak{T}_N\mathfrak{R}_N$ acts on both $\Di_N$ and $\Dio_N$ by homeomorphisms. Furthermore, $T_\Lambda\Dio_{N,\emptyset}=T_\Lambda \mathcal{T}_N^+=\Dio_{N,\Lambda}$ for all $\Lambda\in\PP_N$; in other words, $\mathcal{T}_N^+$ is a fundamental domain for  the $\mathfrak{T}_N$-action on $\Dio_N$.
\item\label{T_NR_N.trans.form.trans.} Let $\Lambda\in\PP_N$ and $\sigma\in S_N$. Set $g=T_\Lambda R_\sigma$, and let $J_g(\theta)=J_g(s,\theta)$ be the element defined in \eqref{J_g}. Then 
\begin{align*}
\LN_N\big(s,g(w,a,\theta)\big)=J_g(\theta)\LN_N(s,w,a,\theta)
\end{align*}
for all $s\in \C\smallsetminus P_N$ and all $(w,a,\theta)\in\Di_N$.
\item\label{T_NR_N.trans.form.hol.} For each $\theta\in\R^{N}$ and each $\Lambda\in\PP_N$, the map $(s,w,a)\mapsto\LN_N(s,w,a,\theta)$ defines a holomorphic function on $(\C\smallsetminus P_N)\times p(\Dio_{N,\Lambda})$, with $p$ as described in \eqref{proj.p}.
\end{enumerate}
\end{proposition}

\begin{proof}
We first prove \eqref{T_NR_N.trans.form.act.}. In view of Lemmas \ref{lemma.R_sigma} and \ref{lemma.T_Lambda}, it verifies that $\pi T_\Lambda R_\sigma(\delta)=\pi(\delta)\in\Co^\circ$ for all $\Lambda\in\PP_N$, $\sigma\in S_N$, and $\delta\in\Di_N$. On the other hand, we have that both $[\C\smallsetminus i\cdot\R^*]^N$ and $[\C\smallsetminus i\cdot\R]^N$ are invariant under the right action of the matrices $r(\sigma)$ and $d(\Lambda)$ ($\Lambda\in\PP_N$, $\sigma\in S_N$). Then the fact that $\mathfrak{T}_N\mathfrak{R}_N$ acts on both $\Di_N$ and $\Dio_N$ follows immediately from \eqref{Di_N} and \eqref{interiorDi}. The last assertion in \eqref{T_NR_N.trans.form.act.} is a consequence of Lemma \ref{bef.projection.lemma} \eqref{bef.projection.lemma.Di}. Indeed, the latter implies that the $T_\Lambda$ permute the connected components of $\Dio_N$.

\eqref{T_NR_N.trans.form.trans.} is a consequence of Proposition~\ref{LN}, Lemma~\ref{lemma.R_sigma}, and Lemma~\ref{lemma.T_Lambda}. Finally, to prove \eqref{T_NR_N.trans.form.hol.}, first note that \eqref{T_NR_N.trans.form.act.}, \eqref{T_NR_N.trans.form.trans.}, and Proposition \ref{LN} \eqref{LN.restr.} imply that
\begin{align*}
\LN_N(s,w,a,\theta)=J_{T_\Lambda}(\theta)^{-1}\zeta_N(s,T_\Lambda(w,a,\theta)) \qquad \big(s\in\C\smallsetminus P_N,\,(w,a)\in p(\Dio_{N,\Lambda})\big).
\end{align*}
Hence the map $(s,w,a)\mapsto \LN_N(s,w,a,\theta)$ defines a holomorphic function since $(s,w,a)\mapsto(s,pT_\Lambda(w,a,\theta))$ is a smooth change of variables and the map $(s,w,a)\mapsto \zeta_N(s,w,a,\theta')$ defines a holomorphic function for any $\theta'\in\R^N$.
\end{proof}

An inspection of the denominator of the test function $F_N$ in \eqref{F_N} shows us that, for any $(w,a,\theta)\in\DD_N$, the $u$-singularities of $F_N(u,w,a,\theta)$ are poles located at 
\begin{align}\label{u-poles.F_N}
\mathcal{P}(a,\theta):=\Big\{u\in\C\,\Big|\,u=u_\ell(m):=\frac{2\pi i}{a_\ell}(m+\theta_\ell), \quad \ell\in a^{-1}[\C^*], \quad m\in\Z\Big\}.
\end{align}
Thus, if we fix any $\ell\in a^{-1}[\C^*]$, they lie along the same straight line passing through the origin as $m$ varies. 

For any nonzero angle $\psi\in[-\pi,\pi]$, recall the set $\mathcal{P}_\psi(a,\theta)$ of all nonzero elements $u\in\mathcal{P}(a,\theta)$ such that either $\arg(u)\in[0,\psi)$ if $\psi>0$, or $\arg(u)\in[\psi,0)$ otherwise. Also, we define 
\begin{align}\label{Di_N(psi)}
\Di_N(\psi):=\Di_N\,\cap\, M_{\e^{i\psi}}^{-1}[\Di_N]\,\cap\,\bigcap_{0<\mathrm{sgn}(\psi)\cdot t<|\psi|}(\pi M_{\e^{it}})^{-1}[\Co].
\end{align}
Thus, for example, $\Di_N(0)=\Di_N$.

\begin{proposition}\label{transf.form.rot.prop.}
Let $N$ be a positive integer, and let $\alpha\in\C^*$. Take $\psi\in[-\pi,\pi]$ such that $\psi\in \arg(\alpha)+2\pi\Z$, and write $\alpha^{-s}=\e^{-s(\log|\alpha|+i\psi)}$. Then, for all $s\in\C\smallsetminus P_N$ and all $(w,a,\theta)\in\Di_N(\psi)$, we have 
\begin{align}\label{transf.form.rot.}
\LN_N\big(s,M_{\alpha}(w,a,\theta)\big)=\alpha^{-s}\Big[\LN_N(s,w,a,\theta)+\rho_N^\psi(s,w,a,\theta)/\Gamma(s)\Big],
\end{align}
where $\rho_N^\psi(s,w,a,\theta)$ is the $s$-analytic continuation of the function 
\begin{align} \label{rho_N} 
\rho_N^\psi(s,w,a,\theta):=\mathrm{sgn}(-\psi)2\pi i\lim_{R\to\infty}\sum_{\substack{u_0\in\mathcal{P}_\psi(a,\theta)\\ |u_0|<R}}\res_{u=u_0}\big(F_N(u,w,a,\theta)u^{s-1}\big)
\end{align}
holomorphic for $\re(s)<1$. Here we use the principal branch of the logarithm to define $u^{s-1}$. 
\end{proposition}

\begin{proof}
It is easy to see that we can divide the proof into two stages by writing $\alpha$ in polar coordinates. Hence, we will first assume that $\alpha$ is a positive real number, and then we will assume that it is a complex number of norm 1.

Suppose that $\alpha$ is a positive real number, so $\psi=0$. Then, in view of Lemma~\ref{lemma.M_alpha}, we obtain
\begin{align*}
\LN_N\big(s,M_{\alpha}(w,a,\theta)\big)=\alpha^{-s}\LN_N(s,w,a,\theta) \qquad\qquad (\alpha\in\R_+)
\end{align*} 
by using the change of variable $\alpha u=u'$ directly in the definition of $\LN_N$ (see Proposition~\ref{LN}). This proves the proposition in the first case.

Suppose that $\alpha$ lies in the unit circle and $\alpha\not=1$, so $\alpha=\e^{i\psi}$ with $\psi\not=0$. Also, suppose that $\re(s)<1$. For any $R>0$, let $C(R,\psi)$ be the counterclockwise oriented circular arc having endpoints at $R$ and $R\e^{i\psi}$. Recall the integral $\IN_N$ defined in \eqref{Integral}, which satisfies
\begin{align}
& \nonumber \IN_N(s,w,a,\theta)=\Gamma(s)(\e(s)-1)\LN_N(s,w,a,\theta) \qquad \mathrm{and} \\
& \label{I_N&L_N} \IN_N(s,M_{\e^{i\psi}}(w,a,\theta))=\Gamma(s)(\e(s)-1)\LN_N(s,M_{\e^{i\psi}}(w,a,\theta))
\end{align}
since both $(w,a,\theta)$ and $M_{\e^{i\psi}}(w,a,\theta)$ belong to $\Di_N$. Take any $\epsilon\in\E$, and write $\epsilon'=(\epsilon M_{\e^{i\psi}})(w,a,\theta)$ for simplicity. In view of Lemma~\ref{lemma.M_alpha}, we can use $\epsilon'$ to compute both integrals in \eqref{I_N&L_N}. Then, using Lemma~\ref{lemma.M_alpha} again, an elementary computation shows that 
\begin{align*}
& \IN_N(s,M_{\e^{i\psi}}(w,a,\theta))-\e^{-i\psi s}\IN_N(s,w,a,\theta)\\
& = \e^{-i\psi s}(1-\e(-s))\Big(\int_{\epsilon'\e^{i\psi}}^{\infty \e^{i\psi}}-\int_{\epsilon'}^\infty-\int_{C(\epsilon',\psi)}\Big)F_N(u,w,a,\theta)u^{s-1}du
\end{align*}
if $\psi<0$, where the argument of $u$ is considered to lie in the interval $(0,2\pi]$ to define $u^{s-1}$. By the same token, we have 
\begin{align*}
& \IN_N(s,M_{\e^{i\psi}}(w,a,\theta))-\e^{-i\psi s}\IN_N(s,w,a,\theta)\\
& = \e^{-i\psi s}(1-\e(s))\Big(\int_{\epsilon'}^\infty-\int_{\epsilon'\e^{i\psi}}^{\infty \e^{i\psi}}-\int_{C(\epsilon',\psi)}\Big)F_N(u,w,a,\theta)u^{s-1}du
\end{align*}
if $\psi>0$, where now the argument of $u$ is considered to lie in $[0,2\pi)$. We remark that no poles of $F_N(u,w,a,\theta)$ lie on the integration contours since $(w,a,\theta)\in\Di(\psi)$.

Now our objective is to use the residue theorem in order to compute the right-hand sides of the last two equations. Take an increasing sequence $\{R_n\}_{n\in\N}$ of positive real numbers approaching to $+\infty$ such that all circles of radius $R_n$ ($n\in\N$) centered at the origin are away from the $u$-singularities of $F_N(u,w,a,\theta)$ (see \eqref{u-poles.F_N}). We claim that
\begin{align*}
\lim_{n\rightarrow\infty}\int_{C(R_n,\psi)}F_N(u,w,a,\theta)u^{s-1}du=0,
\end{align*}
where the argument of $u$ in $u^{s-1}$ is considered to lie either in $(0,2\pi]$ if $\psi<0$, or in $[0,2\pi)$ otherwise. Indeed,  there exists a positive real number $K=K(\psi,w,a,\theta)$, not depending neither on $t$ nor on $n$, such that 
\begin{align*}
|F_N(R_n\e^{it},w,a,\theta)|\leq K\cdot \e^{-R_n\cdot f_{w,a}(t)}, \qquad f_{w,a}(t):=\re\big((\pi M_{\e^{it}})(w,a,\theta)\big),
\end{align*}
for all $t\in\R$ with $0\leq\mathrm{sgn}(\psi)\cdot t \leq|\psi|$. Suppose that $\psi<0$. Hence the zeroes of $\re(\e^{it}a_\ell)$ give us a subdivision 
\begin{align*}
\psi=t_0<t_1<\dots<t_{d-1}<t_d=0
\end{align*} 
of $[\psi,0]$ such that 
\begin{align*}
\e^{-it}(\pi M_{\e^{it}})(w,a,\theta)=w-\Tr\big(a,(\e^{it}a)^{-1}[-\Co]\big)
\end{align*}
is constant and non-zero on each  $(t_{i-1},t_{i}]$ as a function of $t$. This implies that $f''_{w,a}(t)=-f_{w,a}(t)$ on each $(t_{i-1},t_{i})$. Therefore, in view of \eqref{Di_N(psi)}, we conclude that $f_{w,a}(t)$ is concave on each $(t_{i-1},t_{i})$, and that it can only vanish at the $t_i$ ($1\leq i\leq d-1$). As a consequence, we can construct a polygonal chain in the plane lying below the graph of $f_{w,a}:[\psi,0]\to\R$ in order to estimate the integral in the right hand side of
\begin{align*}
\Big|\int_{C(R_n,\psi)}F_N(u,w,a,\theta)u^{s-1}du\Big|\leq K\cdot R_n^s\cdot\int_\psi^0 \e^{-R_n\cdot f_{w,a}(t)}\e^{-\im(s)\cdot t}dt.
\end{align*}
In fact, this integral turns out to be $O(R_n^{-1})$ as $n\to+\infty$. Then our claim follows by taking the limit as $n\to+\infty$. Similarly, we obtain the same result when $\psi>0$.

The residue theorem implies that 
\begin{align*}
& \IN_N(s,M_{\e^{i\psi}}(w,a,\theta))-\e^{-i\psi s}\IN_N(s,w,a,\theta)\\
& = \e^{-i\psi s}(\e(s)-1) \mathrm{sgn}(-\psi)2\pi i\lim_{R\to\infty} \sum_{\substack{u_0\in\mathcal{P}_\psi(a,\theta)\\ |u_0|<R}}\res_{u=u_0}\big(F_N(u,w,a,\theta)u^{s-1}\big)
\end{align*}
for $\re(s)<1$, where the argument of $u$ in $u^{s-1}$ is considered to lie in $(-\pi,\pi]$. From the proof of Proposition~\ref{LN}~\eqref{LN.merom.}, we know that the left-hand side of this equation is actually an entire function of $s$ having zeroes at the integers greater than $N$, and now we have proved that it also has zeroes at the non-positive integers. Then $\rho_N^\psi(s,w,a,\theta)$, as defined in \eqref{rho_N}, is holomorphic for $\re(s)<1$, and it can be analytically continued to $\C\smallsetminus P_N$. Finally, dividing the last equation by $\Gamma(s)(\e(s)-1)$, and using \eqref{I_N&L_N}, the transformation formula \eqref{transf.form.rot.} holds for all $s\in\C\smallsetminus P_N$. 
\end{proof}

Before addressing Proposition \ref{prop.fixed.points}, we describe a handy subset of the domain 
\begin{align*}
\D_N=\bigcap_{\alpha\in\C^*}(\pi M_\alpha)^{-1}[\Co]
\end{align*}
given in \eqref{D.final}. In fact, we define
\begin{align}\label{D^*}
\D_{N,\text{Pol}}:=\big\{(w,a,\theta)\in\DD_N\,\big|\, w=x\cdot a, \quad x\in[0,1]^N\smallsetminus\{0,1\}^N \big\}.
\end{align}
It can be easily seen that $\D_{N,\text{Pol}}\subseteq \D_N$ by using the convexity of $\Co$. Indeed, for any $\alpha\in\C^*$ and any $(x\cdot a,a,\theta)\in\D_{N,\text{Pol}}$, we have
\begin{align*}
(\pi M_\alpha)(x\cdot a,a,\theta)=\sum_{\ell\in(\alpha a)^{-1}[\Co]}x_\ell(\alpha a_\ell)-\sum_{\ell\in(\alpha a)^{-1}[-\Co]}(1-x_\ell)(\alpha a_\ell) \ \in \ \Co
\end{align*}
since $x\notin\{0,1\}^N$, \ie $x$ is not a vertex of the hypercube $[0,1]^N$.

\begin{proof}[Proof of Proposition \ref{prop.fixed.points}]
Let $\Lambda\in\PP_N$, $\sigma\in S_N$, and $\alpha\in\C^*$. Set $g:=T_\Lambda R_\sigma M_\alpha$ and $A:=r(\sigma)d(\Lambda)$. We have to find the solutions $(w,a,\theta)\in\DD_N$ of the equation $g(w,a,\theta)=(w,a,\theta)$, so we must solve
\begin{align}\label{Eq.fix.points}
(w,a,\theta)=\big(\alpha w-\alpha\Tr(a,\sigma^{-1}[\Lambda])\,,\,\alpha aA\,,\,\theta A\big).
\end{align}
It can be verified that $a=\alpha aA$ and $\theta=\theta A$ if and only if 
\begin{align}\label{Eq.fix.points.coor.}
a_\ell=\alpha(-1)^{\ind_\Lambda(\ell)} a_{\sigma^{-1}(\ell)} \qquad \text{and} \qquad 
\theta_\ell=(-1)^{\ind_\Lambda(\ell)} \theta_{\sigma^{-1}(\ell)} \qquad (\forall\ell\in P_N),
\end{align}
where $\ind_\Lambda$ denotes the usual indicator function (see \eqref{indicator.funct.}). If \eqref{Eq.fix.points.coor.} holds, an elementary computation shows that $\Tr(\alpha a,\sigma^{-1}[\Lambda])=-\Tr(a,\Lambda)$ and $\Tr(\theta,\sigma^{-1}[\Lambda])=-\Tr(\theta,\Lambda)$.

Suppose that $\alpha=1$. Then it readily follows that \eqref{Eq.fix.points} holds if and only if $\Tr(a,\Lambda)=0$ and $a, \theta\in E_1[A]$. This proves the first statement in \eqref{prop.fixed.points.alpha=1}. To prove the last one, suppose that \eqref{Eq.fix.points} holds. Then Proposition \ref{T_NR_N.trans.form.} \eqref{T_NR_N.trans.form.trans.} implies $(1-J_g(c\theta))\LN_N(s,w,a,c\theta)=0$  for all $s\in\C\smallsetminus P_N$ and all $c\in\R$ close enough to 1. Therefore $J_g(c\theta)=1$, and so 
\begin{align*}
\e\big(c\Tr(\theta,\Lambda)\big)=(-1)^{|\Lambda|},
\end{align*}
for all $c\in\R$ close enough to 1. But this is possible only if $|\Lambda|$ is even and $\Tr(\theta,\Lambda)=0$.

Suppose that $\alpha\not=1$. If \eqref{Eq.fix.points.coor.} holds, then 
\begin{align*}
\frac{(\alpha-1)}{2}&\Tr(a,P_N)\\
&=\frac{1}{2}\Big[\alpha\Tr(a,P_N)-\Tr(a,P_N)\Big]\\
&=\frac{\alpha}{2}\Big[\Tr(a,P_N)+\Tr(a,\sigma^{-1}[\Lambda])-\Tr(a,\sigma^{-1}[P_N\smallsetminus \Lambda])\Big] \qquad\qquad [\text{by \eqref{Eq.fix.points.coor.}}]\\
&=\frac{\alpha}{2}\Big[\Tr(a,P_N)+\Tr(a,\sigma^{-1}[\Lambda])-\Tr(a,P_N)+\Tr(a,\sigma^{-1}[\Lambda])\Big] \qquad [\text{by \eqref{Tr.prop.sets.minus}}]\\
&=\alpha \Tr(a,\sigma^{-1}[\Lambda]).
\end{align*}
Hence \eqref{Eq.fix.points} holds if and only if $w=\Tr(a,P_N)/2$, $a\in E_{\alpha^{-1}}[A]$, and $\theta\in E_{1}[A]$, and we have inclusions $\DD_N^g\subseteq\D_{N,\text{Pol}}\subseteq\D_N$ (see \eqref{D^*}). Finally, the last statement in \eqref{prop.fixed.points.any.alpha} amounts to say that any eigenvalue of $A$ is a root of unity, which is clear.
\end{proof}

\section{The function $\LN$}

In this final section we prove Theorems \ref{Thm.constr.L} and \ref{Thm.transf.form.}. Then Corollaries \ref{Coro.trans.form.} and \ref{Coro.trans.form.fix.points} follow by elementary arguments of complex analysis.

\subsection{On Theorem \ref{Thm.constr.L}} We start by addressing the extension of $\LN_N$ ($N\geq1$) to the domain $(\C\smallsetminus P_N)\times \Dit_N$, where $\Dit_N=\pi^{-1}[\Co]$ as defined in \eqref{Dit}. In fact, since every element $\delta\in\Di_N$ satisfies $\pi(\delta)\in\Co^\circ$ by definition (see \eqref{Di_N}), we have $\Di_N\subseteq\Dit_N$. To get the desired extension, we first prove that every element in $\Dit_N$ can be approximated by elements in $\Di_N$.

\begin{lemma}\label{Dit.approx}
Let $N$ be a positive integer, and let $(w,a,\theta)\in\Dit_N$. There exists $\omega_0<0$ such that
\begin{align*}
\omega_0\leq\omega<0 \quad\Longrightarrow\quad M_{\e^{i\omega}}(w,a,\theta)\in\Di_N \qquad (\forall \, \omega\in\R).
\end{align*}
\end{lemma}

\begin{proof}
Note that $\pi(w,a,\theta)\in\Co$. We choose $\omega_0<0$ such that 
\begin{align*}
\e^{i\omega}\pi(w,a,\theta)\in\Co^\circ \qquad\text{and}\qquad \e^{i\omega}a\in[\C\smallsetminus i\cdot\R^*]^N \qquad\qquad \big(\forall\,\omega\in[\omega_0,0)\big).
\end{align*}
Consequently we have
\begin{align*}
a^{-1}[-\Co]=(\e^{i\omega}a)^{-1}[-\Co^\circ]=(\e^{i\omega}a)^{-1}[-\Co] \qquad\qquad \big(\forall\,\omega\in[\omega_0,0)\big),
\end{align*}
and so $(\pi M_{\e^{i\omega}})(w,a,\theta)=\e^{i\omega}\pi(w,a,\theta)$ lies in $\Co^\circ$ for all $\omega\in[\omega_0,0)$. In view of \eqref{Di_N}, this ends the proof.
\end{proof}

\begin{proof}[Proof of Theorem \ref{Thm.constr.L}]
Recall the function $\LN_N:(\C\smallsetminus P_N)\times \Dit_N\to\C$ defined in \eqref{L.final} by
\begin{align*}
\LN_N(s,w,a,\theta):=\lim_{\substack{\omega\to0\\ \omega<0}}\LN_{N,\epsilon}\big(s,M_{\e^{i\omega}}(w,a,\theta)\big).
\end{align*}
In view of Proposition~\ref{LN} and Lemma~\ref{Dit.approx}, we know that the right-hand side of the last equation is actually independent of the choice of $\epsilon\in\E$. Therefore the next step is to show that the limit exists.

Let $(w,a,\theta)\in\Dit_N$. Lemma~\ref{Dit.approx} implies that $M_{\e^{i\omega}}(w,a,\theta)\in\Di_N$ for all negative $\omega$ close enough to 0 (abbr. for all $\omega\to0^-$). Since $(\e^{i\omega}a)^{-1}[0]=a^{-1}[0]$, Proposition~\ref{LN}~\eqref{LN.proj.} allows us to assume $a^{-1}[0]=\emptyset$. Consequently, Proposition~\ref{T_NR_N.trans.form.}~\eqref{T_NR_N.trans.form.act.} implies the existence of $\Lambda\in\PP_N$ such that $(T_\Lambda M_{\e^{i\omega}})(w,a,\theta)\in \mathcal{T}_N^+$ for all $\omega\to0^-$. Then, in view of Proposition~\ref{T_NR_N.trans.form.}~\eqref{T_NR_N.trans.form.trans.}, Proposition~\ref{LN}~\eqref{LN.restr.}, and the holomorphicity of $\zeta_N$ on $(\C\smallsetminus P_N)\times p(\Ds_N)$ (see \eqref{series anal. cont.}), we have
\begin{align}\label{extension.limit}
\nonumber\lim_{\substack{\omega\to0\\ \omega<0}}\LN_N\big(s,M_{\e^{i\omega}}(w,a,\theta)\big)&=(-1)^{|\Lambda|}\e\big(-\Tr(\theta,\Lambda)\big)\lim_{\substack{\omega\to0\\ \omega<0}}\zeta_N\big(s,(T_\Lambda M_{\e^{i\omega}})(w,a,\theta)\big)\\
&=(-1)^{|\Lambda|}\e\big(-\Tr(\theta,\Lambda)\big)\zeta_N\big(s,T_\Lambda (w,a,\theta)\big)
\end{align}
for all $s\in\C\smallsetminus P_N$. This proves that the limit exists and also proves \eqref{Thm.constr.L.mero.s}. \eqref{Thm.constr.L.mero.swa} was already established in Proposition~\ref{T_NR_N.trans.form.}~\eqref{T_NR_N.trans.form.hol.}, while \eqref{Thm.constr.L.rest.} follows from \eqref{extension.limit} and Proposition~\ref{LN}~\eqref{LN.restr.}.
\end{proof}

\subsection{On Theorem \ref{Thm.transf.form.}} We first use Lemma~\ref{Dit.approx} to show that the transformation formula for $\LN_N$ under the $\mathfrak{T}_N\mathfrak{R}_N$-action on $\Di_N$ can be extended to $\Dit_N$. In fact, Lemmas \ref{lemma.R_sigma} and \ref{lemma.T_Lambda} state that the group $\mathfrak{T}_N\mathfrak{R}_N$ is contained in $\Aut(\DD_N)_\pi\cap \mathrm{C}(\mathfrak{M}_N)$. Then $\mathfrak{T}_N\mathfrak{R}_N$ acts on $\Dit_N$ and
\begin{align}
\nonumber &\LN_N\big(s,(T_\Lambda R_\sigma)(w,a,\theta)\big)=\lim_{\substack{\omega\to0\\ \omega<0}}\LN_N\big(s,(M_{\e^{i\omega}}T_\Lambda R_\sigma)(w,a,\theta)\big)\\
\nonumber &=\lim_{\substack{\omega\to0\\ \omega<0}}\LN_N\big(s,(T_\Lambda R_\sigma M_{\e^{i\omega}})(w,a,\theta)\big)
=J_{T_\Lambda R_\sigma}(\theta) \lim_{\substack{\omega\to0\\ \omega<0}}\LN_N\big(s,M_{\e^{i\omega}}(w,a,\theta)\big)\\
\label{Transf.form.T_NR_N.final}&=J_{T_\Lambda R_\sigma}(\theta) \LN_N(s,w,a,\theta)
\end{align}
for all $s\in\C\smallsetminus P_N$, $(w,a,\theta)\in\Dit_N$, $\Lambda\in\PP_N$, and $\sigma\in S_N$. 

The fact that $\mathfrak{T}_N\mathfrak{R}_N$ also acts on $\D_N$ is an immediate consequence of the inclusion $\mathfrak{T}_N\mathfrak{R}_N\subseteq \Aut(\DD_N)_\pi\cap \mathrm{C}(\mathfrak{M}_N)$, while the $\mathfrak{M}_N$-action on $\D_N$ is guaranteed by definition. Hence \eqref{Transf.form.T_NR_N.final} also holds if we replace $\Dit_N$ by $\D_N$. However, proving that we can transfer the transformation formula for $\LN_N$ under $\mathfrak{M}_N$-transformations is more subtle. In fact, we would like to use Proposition \ref{transf.form.rot.prop.} together with Lemma \ref{Dit.approx} as before, but then we must still take care of the set  
\begin{align*}
\Di_N(\psi)=\Di_N\,\cap\, M_{\e^{i\psi}}^{-1}[\Di_N]\,\cap\,\bigcap_{0<\mathrm{sgn}(\psi)\cdot t<|\psi|}(\pi M_{\e^{it}})^{-1}[\Co] \qquad\qquad (\psi\in[-\pi,\pi])
\end{align*}
given in \eqref{Di_N(psi)}. To this end we define, for each nonzero angle $\psi\in[-\pi,\pi]$, the set

\begin{align}\label{Dit_N(psi)}
\Dit_N(\psi):=\bigcap_{0\leq\mathrm{sgn}(\psi)\cdot t\leq|\psi|}(\pi M_{\e^{it}})^{-1}[\Co] \qquad \text{and} \qquad \Dit_N(0):=\Dit_N.
\end{align}

\begin{lemma}\label{Dit(psi).approx}
Let $N$ be a positive integer, let $\psi\in[-\pi,\pi]$, and let $(w,a,\theta)\in\Dit_N(\psi)$. There exists $\omega_0<0$ such that
\begin{align*}
\omega_0\leq\omega<0 \quad\Longrightarrow\quad M_{\e^{i\omega}}(w,a,\theta)\in\Di_N(\psi) \qquad (\forall \, \omega\in\R).
\end{align*}
\end{lemma}

\begin{proof}
Using Lemma \ref{Dit.approx}, we choose $\omega_0<0$ such that both $M_{\e^{i\omega}}(w,a,\theta)$ and $M_{\e^{i\psi}}M_{\e^{i\omega}}(w,a,\theta)$ lie in $\Di_N$ for all $\omega\in\R$ with $\omega_0\leq \omega<0$. In particular, we have that $M_{\e^{i\omega}}(w,a,\theta)$ and $M_{\e^{i\psi}}M_{\e^{i\omega}}(w,a,\theta)$ lie in $\pi^{-1}[\Co^\circ]$ for all such $\omega$ (see \eqref{Di_N}). Since $\re[\pi]$ is continuous (see Proposition \ref{projection.lemma} \eqref{projection.lemma.Re[pi]}), the above implies that $M_{\e^{i\omega}}(w,a,\theta)$ and $M_{\e^{i\psi}}M_{\e^{i\omega}}(w,a,\theta)$ lie in $\pi^{-1}[\Co]$ for all $\omega\in\R$ with $\omega_0\leq \omega\leq0$.

Suppose that $\psi>0$, and let $\omega$ be a real number such that $\omega_0\leq \omega<0$. The above paragraph shows that $M_{\e^{i\omega}}(w,a,\theta)\in (\pi M_{\e^{it}})^{-1}[\Co]$ for all $t\in[0,-\omega]$. For $t\in(-\omega,\psi]$ we obtain the same result by using the inclusion $(w,a,\theta)\in\Dit_N(\psi)$.

Suppose that $\psi<0$, and let $\omega$ be as above. In the first paragraph we proved that $M_{\e^{i\omega}}(w,a,\theta)\in (\pi M_{\e^{it}})^{-1}[\Co]$ for all $t\in[\psi,\psi-\omega]$. For $t\in(\psi-\omega,0]$ we obtain the same result by using again the inclusion $(w,a,\theta)\in\Dit_N(\psi)$. The case $\psi=0$ is Lemma \ref{Dit.approx}.
\end{proof}

\begin{proof}[Proof of Theorem \ref{Thm.transf.form.}]
Let $s\in\C\smallsetminus P_N$, $(w,a,\theta)\in\D_N$, and $\alpha\in\C^*$. Set $\psi:=\arg(\alpha)$. The convexity of $\Co$ implies that the set $(\pi M_\beta)^{-1}[\Co]$ ($\beta\in\C^*$) depends only on the argument of $\beta$. Then, from \eqref{Dit_N(psi)}, we see that 
\begin{align*}
\D_N:=\bigcap_{\alpha\in\C^*}(\pi M_{\alpha})^{-1}[\Co]=\bigcap_{\eta\in(-\pi,\pi]}\Dit_N(\eta).
\end{align*}
Hence by Lemma \ref{Dit(psi).approx} we know that $M_{\e^{i\omega}}(w,a,\theta)\in \Di_N(\psi)$ for all $\omega<0$ close enough to 0. Therefore, invoking Proposition \ref{transf.form.rot.prop.} and taking the limit as $\omega\to0^-$, we obtain
\begin{align}\label{Eq.almost.final}
\lim_{\substack{\omega\to0\\ \omega<0}}\rho_N^\psi\big(s,M_{\e^{i\omega}}(w,a,\theta)\big)=\Gamma(s)\Big[\alpha^{s}\LN_N\big(s,M_\alpha(w,a,\theta)\big)-\LN_N(s,w,a,\theta)\Big],
\end{align}
where $\alpha^{s}:=\e^{s(\log|\alpha|+i\psi)}$. 

Suppose that $\re(s)<1$. As it can be seen in \eqref{u-poles.F_N}, the $u$-poles of the function $F_N(u,M_{\e^{i\omega}}(w,a,\theta))u^{s-1}$ are actually the $u$-poles of $F_N(u,w,a,\theta)u^{s-1}$ times $\e^{-i\omega}$. Then an elementary computation shows that the respective residues differ by a factor of $\e^{-i\omega s}$ for all $\omega\to0^-$. On the other hand, since $M_{\e^{i\omega}}(w,a,\theta)\in \Di_N(\psi)$ for all $\omega\to0^-$, we have $\mathcal{P}_\psi(\e^{i\omega}a,\theta)=\mathcal{P}_\psi(a,\theta)$ for all such $\omega$. Taking the limit as $\omega\to0^-$ in equation \eqref{rho_N}, we see that 
\begin{align}\label{Eq.final}
\lim_{\substack{\omega\to0\\ \omega<0}}\rho_N^\psi\big(s,M_{\e^{i\omega}}(w,a,\theta)\big)=\rho_N^\psi(s,w,a,\theta),
\end{align}
where $\rho_N^\psi(s,w,a,\theta)$ is the function in Theorem \ref{Thm.transf.form.}. Hence, we obtain \eqref{Transf.form.LN} by combining \eqref{Transf.form.T_NR_N.final}, \eqref{Eq.almost.final}, and \eqref{Eq.final}.

Now we prove the cocycle property of $J$. Let $k\in\Z$, $\delta=(w,a,\theta)\in\D_N$, and $g_j=T_{\Lambda_j}R_{\sigma_j}M_{\alpha_j}\in G_N$ for $j=1,2$. Set 
\begin{align*}
J_{g_j}(\delta):=J_{g_j}(k,\theta)=(-1)^{|\Lambda_j|}\e\big(\Tr(\theta,\sigma_j^{-1}[\Lambda_j])\big)\alpha_j^{-k} \qquad\qquad (j=1,2).
\end{align*}
We must show that $J_{g_1g_2}(\delta)=J_{g_1}(g_2(\delta))\cdot J_{g_2}(\delta)$. First, using \eqref{Tr.prop.sets.minus} and \eqref{Tr.prop.sets.union}, we get
\begin{align*}
 &\Tr\big(\theta,(\sigma_1\sigma_2)^{-1}[\Lambda_1\oplus\sigma_1(\Lambda_2)]\big)\\
 &=\Tr\big(\theta,(\sigma_1\sigma_2)^{-1}[\Lambda_1]\smallsetminus\sigma_2^{-1}[\Lambda_2]\big)-\Tr\big(\theta,(\sigma_1\sigma_2)^{-1}[\Lambda_1]\cap\sigma_2^{-1}[\Lambda_2]\big)+\Tr\big(\theta,\sigma_2^{-1}[\Lambda_2]\big).
\end{align*}
Then the right-hand side of the last equation equals
\begin{align*}
\Tr\big(\theta d(\sigma_2^{-1}[\Lambda_2]) , (\sigma_1\sigma_2)^{-1}[\Lambda_1]\big)+\Tr\big(\theta,\sigma_2^{-1}[\Lambda_2]\big),
\end{align*} 
since $d(\sigma_2^{-1}[\Lambda_2])$ acts on $\theta$ by changing the signs of the coordinates indexed by $\sigma_2^{-1}[\Lambda_2]$.
Hence \eqref{Lambda&Sigma.relations} implies
\begin{align}\label{Tr.cocycle}
\Tr\big(\theta,(\sigma_1\sigma_2)^{-1}[\Lambda_1\oplus\sigma_1(\Lambda_2)]\big)=\Tr\big(\theta r(\sigma_2)d(\Lambda_2) , \sigma_1^{-1}[\Lambda_1]\big)+\Tr\big(\theta,\sigma_2^{-1}[\Lambda_2]\big).
\end{align}
Therefore, Proposition~\ref{prop.R_sigma.T_Lambda.M_alpha} and \eqref{Tr.cocycle} show that
\begin{align*}
J_{g_1g_2}(\delta)&=(-1)^{|\Lambda_1|}\alpha_1^{-k}(-1)^{|\Lambda_2|}\alpha_2^{-k}\e\big(\Tr\big(\theta,(\sigma_1\sigma_2)^{-1}[\Lambda_1\oplus\sigma_1(\Lambda_2)]\big)\big)\\
&=(-1)^{|\Lambda_1|}\alpha_1^{-k}(-1)^{|\Lambda_2|}\alpha_2^{-k}\e\big(\Tr\big(\theta r(\sigma_2)d(\Lambda_2) , \sigma_1^{-1}[\Lambda_1]\big)\big)\e\big(\Tr\big(\theta,\sigma_2^{-1}[\Lambda_2]\big)\big)\\
&=J_{g_1}(g_2(\delta))\cdot J_{g_2}(\delta).
\end{align*}

Finally, in order to prove \eqref{Rho.Trans.Form.}, we first note that \eqref{Transf.form.LN} implies
\begin{align*}
&\LN_N\big(s,gh(\delta)\big)=J_{gh}(s,\delta)\Big[\LN_N(s,\delta)+\rho_N^{\psi_{gh}}(s,\delta)/\Gamma(s)\Big],\\
&\LN_N\big(s,gh(\delta)\big)=J_g\big(s,h(\delta)\big)\Big[J_h(s,\delta)\Big[\LN_N(s,\delta)+\rho_N^{\psi_h}(s,\delta)/\Gamma(s)\Big]+\rho_N^{\psi_{g}}\big(s,h(\delta)\big)/\Gamma(s)\Big],
\end{align*}
for all $s\in\C\smallsetminus P_N$. Combining these two equations, we obtain
\begin{align}
\nonumber &\left[1-\frac{J_g\big(s,h(\delta)\big)J_h(s,\delta)}{J_{gh}(s,\delta)}\right]\Gamma(s)\LN_N(s,\delta)\\
\label{Rho&cocycle} & \qquad\qquad =\frac{J_g\big(s,h(\delta)\big)J_h(s,\delta)}{J_{gh}(s,\delta)}\rho_N^{\psi_h}(s,\delta)+\frac{J_g\big(s,h(\delta)\big)}{J_{gh}(s,\delta)}\rho_N^{\psi_{g}}\big(s,h(\delta)\big)-\rho_N^{\psi_{gh}}(s,\delta)
\end{align}
for all $s\in\C\smallsetminus P_N$. Note that the function in the left-hand side of the above equation is entire in the variable $s$. Let $k$ be an integer which is not a pole of any of the $\rho_N$ appearing in \eqref{Rho&cocycle}. Then, taking the limit as $s\to k$, the desired result follows from elementary computations.
\end{proof}

\section*{Summary of domains} Since one of the main points of this article is to extend domains, we supply a summary of the most relevant ones for further reference. 

The domain containing all the others is
\begin{align*}
\DD_N:=\big\{(w,a,\theta)\in\C\times \C^{N}\times \R^{N} \, \big| \, a^{-1}[0]\subseteq\theta^{-1}[\R\smallsetminus\Z]\big\},
\end{align*}
which gives a natural set of parameters for the test function $F_N$ to be defined.

The domain
\begin{align*}
\mathcal{T}_N^+=\big\{(w,a,\theta)\in\DD_N \ \big| \ \re(w)>0, \ \re(a_\ell)>0, \ 1\leq \ell\leq N\big\}
\end{align*}
provides a set of parameters for the absolute convergence of the series defining $\zeta_N$.

If we extend $\zeta_N$ by means of its series representation, we arrive at
\begin{align*}
&\Ds_N=\bigcup_{\omega\in(-\pi/2,\pi/2)}\mathcal{T}_N^+(\omega),\\
& \mathcal{T}_N^+(\omega)=\big\{(w,a,\theta)\in\DD_N \, | \, (\e^{i\omega}w,\e^{i\omega}a,\theta)\in \mathcal{T}_N^+\big\} \qquad \Big(-\frac{\pi}{2}<\omega<\frac{\pi}{2}\Big).
\end{align*}

If we extend $\zeta_N$ by using its integral representation, we obtain 
\begin{align*}
\Di_N=\big\{(w,a,\theta)\in\DD_N \, \big| \,  \pi(w,a,\theta)\in\Co^\circ \quad \text{and} \quad a\in[\C\smallsetminus i\cdot\R^*]^N\big\},
\end{align*}
where $\pi$ and $\Co$ are defined in \eqref{pi} and \eqref{Co} respectively. Then we add some limit points to get
\begin{align*}
\Dit_N=\big\{(w,a,\theta)\in\DD_N\,|\,\pi(w,a,\theta)\in\Co\big\}.
\end{align*}
Finally, we take the subset 
\begin{align*}
\D_N=\bigcap_{\alpha\in\C^*}(\pi M_\alpha)^{-1}[\Co]
\end{align*}
of $\Dit_N$ stable under the action of $\C^*$ (see \eqref{M_alpha}).

\section*{Acknowledgement} During the preparation of this article I benefited from the facilities of the Institut de Math\'ematiques de Jussieu, the Universit\'e Pierre et Marie Curie, and the Max-Planck-Institut f\"ur Mathematik. I am very grateful to them for their hospitality, and I would like to address special thanks to Pierre Charollois for his generous support and enlightening discussions. Also, I thank the anonymous referee for his/her helpful comments and suggestions.

\end{document}